\documentclass[11pt,reqno]{amsart}
\usepackage[utf8]{inputenc}
\usepackage[T1]{fontenc}
\usepackage[english]{babel}
\usepackage{amsmath,amsfonts,amssymb,amsthm}

\usepackage[colorlinks=true, urlcolor=blue, linkcolor=blue, citecolor=blue, pdftex]{hyperref}

\usepackage[a4paper,body={160mm,235mm},centering]{geometry}

\usepackage{enumerate}

\usepackage{tikz}
\usetikzlibrary{shapes}
\usepackage{titlesec}
\usepackage{algorithmicx}
\usepackage{algpseudocode}
\usepackage{algorithm}
\usepackage{color}

\newcommand{\dotafter}[1]{#1.}
\titleformat{\section}[hang]
{\normalfont\large\bfseries}{\thesection.}{.5em}{\dotafter}[]
\titleformat{\subsection}[runin]
{\normalfont\bfseries}{\thesubsection.}{.4em}{}[.]
\titlespacing*{\subsection}{0pt}{3ex plus 1ex minus .2ex}{1em}
\titleformat{\paragraph}[runin]{\normalfont\bfseries}{\theparagraph.}{.4em}{}[.]
\titlespacing*{\paragraph}{0pt}{2.5ex plus 1ex minus .2ex}{1em}

\newcommand{\bmu}{\bar{\mu}}

\newcommand{\bU}{\bar{U}}
\newcommand{\bX}{\bar{X}}

\newcommand{\tU}{\tilde{U}}

\newcommand{\tX}{\tilde{X}}

\newcommand{\hmu}{\hat{\mu}}

\newcommand{\E}{\mathbb{E}}

\newcommand{\R}{\mathbb{R}}

\newcommand{\mP}{\mathcal{P}}

\newcommand{\p}{\partial}
\newcommand{\Lip}{{\rm Lip}}

\newcommand{\tmu}{\tilde{\mu}}

\renewcommand{\p}{\partial}

\newcommand{\s}{\underline{s}}
\newcommand{\ur}{\underline{r}}

\renewcommand{\l}{\left}
\renewcommand{\r}{\right}

\renewcommand{\d}{{\rm d}}

\newcommand{\ce}{\color{black}}

\numberwithin{equation}{section}

\newtheorem{theoreme}{Theorem}[section]
\newtheorem{proposition}[theoreme]{Proposition}

\newtheorem{lemme}[theoreme]{Lemma}

\theoremstyle{definition}
\newtheorem{definition}[theoreme]{Definition}
\newtheorem{hyp}{Assumption}[section]

\theoremstyle{remark}
\newtheorem{remarque}[theoreme]{Remark}

\begin{document}
\title[Particle Systems for Mean reflected SDEs]{Particles Systems and Numerical Schemes for \\ Mean Reflected Stochastic Differential Equations}
\author[Ph. Briand]{Philippe Briand}
\address{Univ. Savoie Mont Blanc, CNRS, LAMA, F-73000 Chambéry, France}
\email{philippe.briand@univ-smb.fr}

\author[P.-É. Chaudru de Raynal]{Paul-Éric Chaudru de Raynal}
\address{Univ. Savoie Mont Blanc, CNRS, LAMA, F-73000 Chambéry, France}
\email{paul-eric.chaudru-de-raynal@univ-smb.fr}

\author[A. Guillin]{Arnaud Guillin}
\address{LM UMR 6620, Univ. Blaise Pascal \& CNRS}
\email{arnaud.guillin@math.univ-bpclermont.fr}

\author[C. Labart]{Céline Labart}
\address{Univ. Savoie Mont Blanc, CNRS, LAMA, F-73000 Chambéry, France}
\email{celine.labart@univ-smb.fr}

\thanks{The authors would like to thank very much Hélène Hibon for her careful reading of a previous version of this work and for her valuable comments and remarks.}

\date{Last modification: 2017-02-28}

\begin{abstract}
This paper is devoted to the study of reflected Stochastic Differential Equations when the constraint is not on the paths of the solution but acts on the law of the solution. These reflected equations have been introduced recently by Briand, Elie and Hu \cite{briand_bsdes_2016} in the context of risk measures. Our main objective is to provide an approximation of solutions to these reflected SDEs with the help of interacting particles systems. This approximation allows to design a numerical scheme for this kind of equations. 

\end{abstract}

\maketitle

\section{Introduction}
In this paper, we are concerned with a special type of reflected Stochastic Differentials Equations (SDEs for short in the sequel) in which the constraint is not directly on the paths of the solution to the SDE as in the usual case but on the law of the solution. Typically, the integral of a given function, say $h$, with respect to the law of the solution to the SDE is asked to be nonnegative. We call Mean Reflected Stochastic Differential Equation (MR-SDE) this kind of reflected SDEs which are described by the following system:
\begin{equation}\label{eq:main}	
	\left\lbrace\begin{split}
	& X_t  =X_0+\int_0^t b(X_s)\d s + \int_0^t \sigma(X_s) \d B_s + K_t, \quad t\geq 0,\\
	& \E[h(X_t)] \geq 0, \quad \int_0^t \E[h(X_s)] \, \d K_s = 0, \quad t\geq 0,
	\end{split}\right.
	\end{equation}
where $b$, $\sigma$ and $h$ are given Lipschitz functions from $\R$ to $\R$ and where $(B_{t}, t\geq 0)$ stands for a standard Brownian motion defined on some complete probability space $(\Omega, \mathcal{F}, \mathbb{P})$. We will always assume that $h$ is nondecreasing and that the law of $X_0$ is such that $\E[h(X_0)] \geq 0$. The solution to~\eqref{eq:main} is the couple of continuous processes $(X,K)$, $K$ being needed to ensure that the constraint is satisfied, in a minimal way according to the last condition namely the Skorokhod condition.

Reflected stochastic differential equations have been widely studied in the literature and we refer to the works \cite{tanaka_stochastic_1979,p._-l._lions_stochastic_1984} for an overview of this theory. As said before, the main particularity comes here from the fact that the constraint acts on the law of the process $X$ rather than on its paths. This kind of processes has been introduced by Briand, Elie and Hu in their backward forms in \cite{briand_bsdes_2016}. In that work, they show that mean reflected backward stochastic processes exist and are uniquely defined by the associated system of equations of the form of \eqref{eq:main} under the same assumptions we use below. The main requirement for uniqueness is to ask the process $K$ to be a deterministic function. Our main objective in this paper is to study the convergence of a particles system approximation of \eqref{eq:main} in order to be able to design numerical schemes for computing solutions to~\eqref{eq:main}.

Due to the fact that the reflection process $K$ depends on the law of the position, we use for the approximation a particle system interacting through the reflection process. In the terminology of McKean-Vlasov processes (see \cite{burkholder_topics_1991} for an overview), the reflection is hence non-linear. To conclude the analogy with McKean-Vlasov processes let us emphasize that the results obtained in the sequel can be extended to the case where the dynamic of the process \eqref{eq:main} depends also on its own law.

\medskip

Our main motivation for studying~\eqref{eq:main} comes from financial problems submitted to risk measure constraints. Given any position $X$, its risk measure $\rho(X)$ can be seen as the amount of own fund needed by the investor to hold the position. Mathematically, a risk measure is a nonincreasing application $\rho: \mathbb{L}^2(\Omega) \ni X \mapsto \rho(X) \in \R$ which is translation invariant in the sense that for any real constant $m$, $\rho(\cdot+m) = \rho(\cdot)-m$. Given a risk measure, the acceptance set $\mathcal{A}$ \emph{i.e.} the set of all acceptable positions is defined as $ \mathcal{A} = \{X:\ \rho(X)\leq 0\}$. A classical example of risk measure is the so called Value at Risk at level $\alpha \in (0,1)$:  ${\rm VAR}_\alpha(X) = \inf\{m:\ \mathbb{P}(m+X<0)\leq \alpha\}$ with acceptance set $\mathcal{A}=\{X:\ \mathbb{P}(X<0)\leq 0\}$. This means that the probability of loss has to be below the level $\alpha$. Another example of risk measure is the following: $\rho(X) = \inf\{m:\ \mathbb{E}[u(m+X)]\geq p\}$ where $u$ is a utility function (concave and increasing) and $p$ is a given threshold. In this case the acceptance set is $\mathcal{A}=\{X:\ \mathbb{E}[u(X)]\geq p\}$: a minimal profit is required. We refer the reader to~\cite{ADEH99} for coherent risk measures and to~\cite{FS02} for convex risk measures.

Suppose now that we are given a portfolio $X$ of assets whose dynamic, when there is no constraint, follows the SDE
\begin{equation*}
\d X_t = b(X_t) \d t + \sigma(X_t) \d B_t, \qquad t\geq 0.
\end{equation*}
For instance, in the Black \& Scholes model, when the strategy of the agent depends on the wealth of the portfolio. Given a risk measure $\rho$, one can ask that $X_t$ remains an acceptable position at each time $t$. In the case of the value at risk, this means that the probability of loss is not too great and, in the case of a risk measure defined by a utility function, a minimal profit is guaranteed. In both examples, the constraint rewrites $\E\left[h(X_t)\right] \geq 0$ for $t\geq 0$ ($h = \mathbf{1}_{\R^+}-(1-\alpha)$ in the case of $\text{VAR}_\alpha$, $h=u-p$ in the utility case).

In order to satisfy this constraint, the agent has to add some cash in the portfolio through the time and the dynamic of the wealth of the portfolio becomes 
\begin{equation*}\label{eq:exempleport2}
\d X_t = b(X_t) \d t + \sigma(X_t) \d B_t +\d K_t, \qquad t\geq 0,
\end{equation*}
where $K_t$ is the amount of cash added up to time $t$ in the portfolio to balance the "risk" associated to $X_t$. Of course, the agent wants to cover the risk in a minimal way, adding cash only when needed: this leads to the Skorokhod condition $\E[h(X_t)] \d K_t = 0$. 

Putting together all conditions, we end up with a dynamic of the form \eqref{eq:main} for the portfolio.

\medskip

The paper is organized as follows. In Section \ref{sec:EUPS}, by letting the coefficients satisfy the usual smoothness assumptions (say Lipschitz continuity) and adding a structural assumption on the function $h$ (say $h$ bi-Lipschitz), we show that the system admits a unique strong solution \emph{i.e.} there exists a unique pair of process $(X,K)$ satisfying system \eqref{eq:main} almost surely, the process $K$ being an increasing and deterministic process. Then, we show that, by adding some regularity on the function $h$, the Stieljes measure $\d K$ is absolutely continuous with respect to the Lebesgue measure and we give the explicit expression of its density.  

Having in mind the analogy with McKean-Vlasov processes, we also show in Section \ref{sec:PMRSDE} that system \eqref{eq:main} can be seen as the limit of an interacting particles system with oblique reflection of mean field type. This could reflect a system of large number of players whose positions are constrained by the mean of the positions of the other players. If all the players have the same dynamic, and if the interaction between the players is on mean field type, we show that there is a propagation of chaos phenomenon so that when the number of players tends to the infinity, the reflection no more depends on the other positions, but only on their statistical distribution. This obviously comes from the law of large number and is what, in fact, exactly happened in the classical McKean-Vlasov setting.

As an application, this result allows to define in Section \ref{sec:NSMRSDE} an algorithm based on this interacting particle system together with a classical Euler scheme which gives a strong approximation of the solution of \eqref{eq:main}. This leads to an approximation error proportional, up to a $\log$ factor, to the number of points of the discretization grid of the time interval (namely of $(\log n/n)^{1/2}$, where $n$ is the number of points of the discretization grid) and on the number of particles (namely $N^{-1/4}$ when the function $h$ is only bi-Lipschitz and $N^{-1/2}$ when the function $h$ is smooth, $N$ standing for the number of particles). Finally, we illustrate in Section \ref{sec:NI} these results numerically.

\section{Existence, uniqueness and properties of the solution}\label{sec:EUPS}

Throughout this paper, we consider the following set of assumptions.

\begin{hyp}\label{H:eds}\hfill
	\renewcommand{\labelenumi}{(\roman{enumi})}
	\begin{enumerate}
		\item The functions $b : \R\longmapsto\R$ and $\sigma: \R\longmapsto\R$ are  Lipschitz continuous ;
		\item The random variable $X_0$ is square integrable.
	\end{enumerate}
\end{hyp}

\begin{hyp}\label{H:h}\hfill
	\renewcommand{\labelenumi}{(\roman{enumi})}
	\begin{enumerate}
		\item The function $h: \R\longmapsto\R$ is an increasing function and there exist $0<m\leq M$ such that
		\begin{equation*}
			\forall x\in\R,\quad \forall y\in\R,\qquad m|x-y|\leq |h(x)-h(y)| \leq M|x-y|.
		\end{equation*}
		\item The initial condition $X_0$ satisfies: $\E[h(X_0)]\geq 0$.
	\end{enumerate}
\end{hyp}

\begin{hyp}\label{H:moments}
	\renewcommand{\labelenumi}{(\roman{enumi})}
	$\exists p>4$ such that $X_0$ belongs to $\mathbb{L}^p$: $\E[|X_0|^p] < \infty$.
\end{hyp}

\begin{hyp}\label{H:smooth}
	\renewcommand{\labelenumi}{(\roman{enumi})}
	The mapping $h$ is a twice continuously differentiable  function with bounded derivatives.
\end{hyp}

We emphasize that existence and uniqueness results hold only under \ref{H:eds} which is the standard assumption for SDEs and \ref{H:h} which is the assumption used in~\cite{briand_bsdes_2016}. The convergence of particle systems require only an additional integrability assumption on the initial condition, namely~\ref{H:moments}. We sometimes add the smoothness assumption \ref{H:smooth} on $h$ in order to improve some of the results.

We first recall the existence and uniqueness result of \cite{briand_bsdes_2016} in the case of SDEs.

\begin{definition}
	 A couple of continuous processes $(X,K)$ is said to be a flat deterministic solution to~\eqref{eq:main} if $(X,K)$ satisfy \eqref{eq:main} with $K$ being a non-decreasing deterministic function with $K_0=0$.
\end{definition}

Given this definition, we have the following result.

\begin{theoreme}[Briand, Elie and Hu, \cite{briand_bsdes_2016}]\label{th:wp}
Under assumptions \ref{H:eds} and \ref{H:h}, the mean reflected SDE \eqref{eq:main} has a unique deterministic flat solution $(X,K)$. Moreover,
\begin{equation*}
	\forall t\geq 0,\quad K_t = \sup_{s\leq t} \inf\{x \geq 0\:\ \E\left[ h(x + U_s )\right] \geq 0\},
\end{equation*}
where $(U_t)_{0\leq t \leq T}$ is the process defined by:
\begin{equation}
	\label{eq:defu}
	U_t = X_0 + \int_0^t b(X_s)\d s \int_0^t\sigma(X_s)\d B_s.
\end{equation}
\end{theoreme}

\begin{proof}
	The proof for the case of backward SDEs is given in~\cite{briand_bsdes_2016}. For the ease of the reader, we sketch the proof for the forward case.

	Let $\tilde{X}$ be a given continuous process such that, for all $t>0$, $\E\left[\sup_{s\leq t}\left|\tX_s \right|^2\right] < +\infty$. We set
	\begin{equation*}
	\tU_t = X_0 + \int_0^t b(\tX_s)\d s + \int_0^t \sigma(\tX_s)\d B_s,
	\end{equation*}
	and define the function $K$ by setting
	\begin{equation}
		\label{eq:dfkt}
	K_t = \sup_{s \leq t } \inf\{x\geq 0:\ \E \left[h(x+\tU_s)\right]\geq 0\}.
	\end{equation}

	The function $K$ being given, let us define the process $X$ by the formula
	\begin{equation*}
	X_t = X_0 + \int_0^t b(\tX_s)\d s + \int_0^t \sigma(\tX_s)\d B_s + K_t,
	\end{equation*}
	Let us check that $(X,K)$ is the solution to~\eqref{eq:main}. By definition of $K$, $\E\left[h(X_t)\right]\geq 0$ and we have, $\d K$ almost everywhere,
	\begin{equation*}
		K_t = \sup_{s\le t}\inf\{x\geq 0:\ \E \left[h(x+\tU_s)\right]\geq 0\} >0,
	\end{equation*}
	so that $\E\left[h(X_t)\right]=\E\left[h\left(\tU_t + K_t\right)\right]=0$ $dK$-a.e. since $h$ is continuous and nondecreasing.

	Next, consider the map $\Xi$ which associates to $\tX$ the solution $X$ of \eqref{eq:main}. Let us show that $\Xi$ is a contraction. Let $\tX$ and $\tX'$ be given, and define $K$ and $K'$ as above, using the same Brownian motion. We have from Cauchy-Schwartz and  Doob inequality
	\begin{equation*}
		\E\left[\sup_{t\leq T} |X_t-X_t'|^2\right] \leq  ( T||b||_{\Lip} +2 ||\sigma||_{\Lip}) \E\left[ \int_0^T |\tX_s-\tX_s'|^2 \d s\right] +2\, \sup_{t\leq T}|K_t-K'_t|^2.
	\end{equation*}
	From the representation~\eqref{eq:dfkt} of the process $K$ and Lemma \ref{lemme:regG}, we have that
	\begin{equation*}
		\sup_{t\leq T}|K_t-K'_t|^2 \leq \frac{M}{m}\, \E \left[\sup_{t\leq T}|\tU_t - \tU'_t|^2\right] \leq C\, \left( T||b||_{\Lip} +2 ||\sigma||_{\Lip}\right) \E \left[\int_0^T |\tX_s-\tX_s'|^2 \d s\right].
	\end{equation*}
	Therefore,
	\begin{equation*}
		\E\left[\sup_{t\leq T}\left|X_t-X_t'\right|^2\right] \leq  C (1+T) \E \left[\int_0^T \left|\tX_s-\tX_s'\right|^2 \d s\right]\leq  C (1+T)T\,  \E\left[\sup_{t\leq T}\left|\tX_t-\tX_t'\right|^2\right].
	\end{equation*}

	Hence, there exists a positive $\mathcal{T}$, depending on on $b$, $\sigma$ and $h$ only, such that for all $T<\mathcal{T}$, the map $\Xi$ is a contraction. We first deduce the existence and uniqueness of the solution on $[0,\mathcal{T}]$ and then on $\R^+$ by iterating the construction.
\end{proof}

\begin{remarque}
	\label{rem:markovK}
Note that from this construction, we deduce that for all positive $r$:
$$K_{t+r} - K_t = \sup_{0\leq s \leq r} \inf \left\{x\geq 0\ :\ \E\left[h\left(x+X_t+\int_t^{t+s} b(X_u)\d u + \int_t^{t+s}\sigma(X_u)\d B_u\right)\right]\right\}.$$
It then follows that the unique solution of \eqref{eq:main} is a Markov process on the space $\R \times \mathcal{P}(\R)$, where $\mathcal{P}(\R)$ the space of probability measures on $\R$.
\end{remarque}

In the following, we make an intensive use of this representation formula of the process $K$. Define the function
\begin{equation}
	\label{eq:defH}
H : \R\times \mathcal{P}(\R) \ni (x,\nu) \mapsto H(x,\nu) = \int h(x+z)\nu(dz).
\end{equation}
We will need also the inverse function in space of $H$ evaluated at $0$ namely:
\begin{equation}\label{eq:}
	\bar G_0 : \mathcal{P}(\R) \ni \nu \mapsto \inf\left\{x\in\R :\ H(x,\nu) \geq 0 \right\},
\end{equation}
as well as $G_0$, the positive part of $G_0$: 
\begin{equation}\label{eq:defG0}
G_0 : \mathcal{P}(\R) \ni \nu \mapsto \inf\left\{x \geq 0\ :\ H(x,\nu) \geq 0 \right\}.
\end{equation}
With these notations, denoting by $(\mu_t)_{0\leq t \leq T}$ the family of marginal laws of $(U_{t})_{0\leq t \leq T}$ we have
\begin{equation}
K_t = \sup_{s\leq t} G_0(\mu_s).\label{eq:repdeKG}
\end{equation}

We start by studying some properties of $H$ and $G_0$.

\begin{lemme}\label{lem:regH} Under \ref{H:h} we have:
\begin{enumerate}[(i)]
\item For all $\nu$ in $\mathcal{P}(\R)$, the mapping $H(\cdot,\nu) : \R \ni x \mapsto H(x,\nu)$ is a bi-Lipschitz function, namely:
\begin{equation}\label{reg:Hspace}
\forall x,y \in \R,\ m|x-y|\leq |H(x,\nu)-H(y,\nu)| \leq M|x-y|
\end{equation}
\item For all $x$ in $\R$, the mapping $H(x,\cdot) : \mathcal{P}(\R) \ni \nu \mapsto H(x,\nu)$ satisfies the following Lipschitz estimate:
\begin{equation}\label{reg:Hmeasure}
\forall \nu,\nu' \in \mathcal{P}(\R),\ |H(x,\nu)-H(x,\nu')| \leq \Big |\int h(x+\cdot) (\d\nu - \d\nu')\Big |.
\end{equation}
\end{enumerate}
\end{lemme}

\begin{proof}
	The proof is straightforward from the definition of $H$ see~\eqref{eq:defH}.
\end{proof}

Note that thanks to Monge-Kantorovitch Theorem, assertion \eqref{reg:Hmeasure} implies that for all $x$ in $\R$, the function $H(x,\cdot)$ is Lipschitz continuous w.r.t. the Wasserstein-1 distance. Indeed, for two probability measures $\nu$ and $\nu'$, the Wasserstein-1 distance between $\nu$ and $\nu'$ is defined by:
$$W_1(\nu,\nu') = \sup_{\varphi\ 1-{\rm Lipschitz}}\left|\int \varphi (\d \nu-\d \nu')\right| = \inf_{X\sim\nu\ ;\ Y\sim \nu'}\E[|X-Y|].$$
Therefore
\begin{equation}
	\label{reg:Hmeasurewass}
\forall \nu,\nu' \in \mathcal{P}(\R),\ |H(x,\nu)-H(x,\nu')| \leq M W_1(\nu,\nu').
\end{equation}

Then, we have the following result about the regularity of $G_0$:

\begin{lemme}\label{lemme:regG}
Under \ref{H:h}, the mapping $G_0 : \mathcal{P}(\R) \ni \nu \mapsto G_0(\nu)$ is Lipschitz-continuous in the following sense:
$$
|G_0(\nu) - G_0(\nu')| \leq \frac{1}{m}  \left\vert\int h(\bar{G}_0(\nu)+\cdot) (\d\nu - \d\nu')\right\vert,
$$
where $\bar{G}_0(\nu)$ is the inverse of $H(\cdot,\nu)$ at point 0. In particular,
\begin{equation}\label{regG:wass}
|G_0(\nu) - G_0(\nu')| \leq \frac{M}{m}  W_1(\nu,\nu').
\end{equation}
\end{lemme}

\begin{proof}
Let $\nu$ and $\nu'$ be two probability measures on $\R$.  From Lipschitz regularity of the positive part, we have
\begin{equation*}
|G_0(\nu) - G_0(\nu')| \leq  |\bar{G}_0(\nu) - \bar{G}_0(\nu')|\\
\end{equation*}
Next, using the bi-Lipschitz in space property of $H$, we get that for any $\eta$ in $\mP(\R)$:
\begin{equation*}
|\bar{G}_0(\nu) - \bar{G}_0(\nu')| \leq  \frac{1}{m} |H(\bar{G}_0(\nu),\eta) - H(\bar{G}_0(\nu'),\eta)|.
\end{equation*}
By definitions of $H$ and $G_0$ we have, for all $\eta$ in $\mathcal{P}(\R)$:  $H(\bar{G}_0(\eta),\eta) = 0$. Hence, by choosing $\eta=\nu'$:
\begin{eqnarray*}
|G_0(\nu) - G_0(\nu')|  &\leq &  \frac{1}{m} |H(\bar{G}_0(\nu),\nu')  - H(\bar{G}_0(\nu'),\nu')|\\
&\leq &\frac{1}{m} |H(\bar{G}_0(\nu),\nu') - H(\bar{G}_0(\nu),\nu)| + |H(\bar{G}_0(\nu),\nu) - H(\bar{G}_0(\nu'),\nu')|\\
&\leq & \frac{1}{m} \left| \int h(\bar{G}_0(\nu)+\cdot) (\d\nu-\d\nu')\right|.
\end{eqnarray*}
The last assertion immediately follows from \eqref{reg:Hmeasurewass}.
%
\end{proof}

We close this section by giving some additional properties of the solution $(X,K)$ of \eqref{eq:main}.

Let $\mathcal{L}$ be the linear partial operator of second order defined by
\begin{equation}\label{eq:defmL}
\mathcal{L}f(x) := b(x)\frac{\p}{\p x}f(x) + \sigma\sigma^*(x)\frac{\p^2}{\p x^2}f(x),
\end{equation}
for any twice continuously differentiable function $f$.
\begin{proposition}\label{prop:densityK}
Suppose that assumptions \ref{H:eds}, \ref{H:h} and \ref{H:smooth} hold and let $(X,K)$ denotes the unique deterministic flat solution to \eqref{eq:main}. Then the Stieljes measure $\d K$ is absolutely continuous with respect to the Lebesgue measure with density
\begin{equation}
k : t \ni \R^{+} \longmapsto \frac{\left(\E \left[\mathcal{L}h(X_t)\right]\right)^-}{\E\left[h'(X_t)\right]} \mathbf{1}_{\E \left[h(X_t)\right] = 0}.
\end{equation}
\end{proposition}

\begin{proof}
	Let $t$ in $[0,T]$. Then, thanks to Itô's Formula we have for all positive $r$:

	\begin{equation}\label{eq:repKinter}
	\E \int_t^{t+r} h'(X_s) \d K_s = \E h(X_{t+r})-\E h(X_t) - \int_t^{t+r}\E \mathcal{L}h(X_s) \d s,
	\end{equation}
	where $\mathcal{L}$ is given by \eqref{eq:defmL}.

	Suppose, at the one hand, that $\E h(X_t)>0$. Then, by the continuity of $h$ and $X$, we get that there exists a positive $\mathcal{R}$ such that for all $r \in [0,\mathcal{R}]$, $E h(X_{t+r})>0$. This implies in particular, from the definition of $K$, that $\d K([t,t+r]) = 0$ for all $r$ in $[0,\mathcal{R}]$.

	At the second hand, suppose that  $\E h(X_t)=0$, then two cases arise. Let us first assume that $\E \mathcal{L}h(X_t) >0$. Hence, we can find a positive $\mathcal{R}'$ such that for all $r \in [0,\mathcal{R}']$, $E \mathcal{L}h(X_{t+r})>0$. We thus deduce from our assumptions and \eqref{eq:repKinter} that $\E h(X_{t+r}) >0$ for all $r$ in $(0,\mathcal{R}']$. Therefore, $\d K([t,t+r]) = 0$ for all $r$ in $[0,\mathcal{R}']$. Suppose next that $E \mathcal{L}h(X_{t+r})\leq 0$. By continuity again, there exists a positive $\mathcal{R}''$ such that for all $r$ in $[0,\mathcal{R}'']$ it holds $E \mathcal{L}h(X_{t+r})\leq 0$. Since $\E h(X_{t+r})$ must be positive on this set, we could have to compensate and $K_{t+r}$ is then positive for all $r$ in $[0,\mathcal{R}'']$. Moreover, the compensation must be minimal \emph{i.e.} such that $\E h(X_{t+r})=0$. Equation \eqref{eq:repKinter} becomes:
	\begin{equation*}
	\E \int_t^{t+r} h'(X_s) \d K_s = - \int_t^{t+r}\E \mathcal{L}h(X_s) \d s,
	\end{equation*}
	on $[0,\mathcal{R}'']$. Dividing both sides by $r$ and taking the limit $r \to 0$ gives (by continuity):
	\begin{equation*}
	\d K_t  = -\frac{\E \mathcal{L}h(X_s)}{\E h'(X_s)} \d t.
	\end{equation*}

	Thus, $\d K$ is absolutely continuous w.r.t. the Lebesgue measure with density:
	\begin{equation}
	k_t = \frac{(\E \mathcal{L}h(X_s))^-}{\E h'(X_s)} \mathbf{1}_{\E h(X_t) = 0}.
	\end{equation}
\end{proof}

\begin{remarque}
	This justifies, at least under the smoothness assumption \ref{H:smooth} on the constraint function $h$, the non-negative hypothesis imposed on $h'$.
\end{remarque}

Finally, we have the following result concerning the moments of the solution of \eqref{eq:main}.

\begin{proposition}\label{prop:polymom}
Suppose that assumptions \ref{H:eds} and \ref{H:h} hold. Then, for all $p \geq 1$, there exists a positive constant $K_p$, depending on $T$, $b$, $\sigma$ and $h$ such that 
\begin{equation*}
	\E\left[\sup\nolimits_{t \leq T} |X_t|^p\right] \leq K_p \left(1+ \E\left[|X_0|^p\right]\right).
\end{equation*}
Moreover, if $X_0\in\mathbb{L}^p$ for some $p\geq 1$, there exists a constant C, depending on $p$, $T$, $b$, $\sigma$ and $h$ such that
\begin{equation*}
	\forall 0\leq s\leq t\leq T, \qquad \E\left[|X_t-X_s|^p\right] \leq C\, |t-s|^{p/2}.
\end{equation*}
\end{proposition}

\begin{proof}
	We have 
	\begin{equation*}
	\E \sup_{t\leq T}|X_t|^p \leq 4^{p-1}\left\{ \E |X_0|^p + \E  \sup_{t \leq T}\left( \int_0^t |b(X_s)| \d s \right)^p + \E \sup_{t \leq T}\left|\int_0^t \sigma(X_s)\d B_s \right|^p + K_T^p\right\}.
	\end{equation*}
	Let us first consider the last term $K_T=\sup_{t \leq T}| G_0(\mu_s)|^p$. From the Lipschitz property of lemma \ref{lemme:regG} of $G_0$, and the definition of the Wasserstein metric we have
	$$
	\forall t\geq 0,\ |G_0(\mu_t)| \leq \frac{M}{m}\, \E[\left|U_t-U_0|\right],
	$$
	since $G_0(\mu_0)=0$ as $\E\left[h(X_0)\right]\geq 0$ and where $U$ is defined by \eqref{eq:defu}. Therefore
	\begin{equation*}
	|\sup_{t \leq T} G_0(\mu_s)|^p  \leq 2^{p-1}\left(\frac{M}{m}\right)^p \left\{ \E \sup_{t \leq T} \left( \int_0^t |b(X_s)| \d s \right)^p + \E \sup_{t \leq T}\left| \int_0^t \sigma(X_s) \d B_s \right|^p\right\},
	\end{equation*}
	and so 
	\begin{equation*}
	\E\left[\sup_{t\leq T}|X_t|^p\right] \leq C(p,M,m) \E\left[ |X_0|^p +  \sup_{t \leq T}\left( \int_0^t |b(X_s)| \d s \right)^p + \sup_{t \leq T}\left|\int_0^t \sigma(X_s)\d B_s \right|^p\right].
	\end{equation*}
	The first part of the result follows from standard computations since $b$ and $\sigma$ are Lipschitz continuous.
	
	For the second part, the key observation is Remark~\ref{rem:markovK}:
	\begin{equation*}
		K_t-K_s = \sup_{s\leq r \leq t} \inf \left\{x\geq 0\ :\ \E\left[h\left(x+X_s+\int_s^{r} b(X_u)\d u + \int_s^{r}\sigma(X_u)\d B_u\right)\right]\right\}.
	\end{equation*}
	Therefore, since $\E\left[h(X_s)\right]\geq 0$, $G_0(X_s)=0$ and we have
	\begin{align*}
		K_t-K_s & = \sup_{s\leq r \leq t} \left| G_0\left( X_s+\int_s^{r} b(X_u)\d u + \int_s^{r}\sigma(X_u)\d B_u\right)-G_0(X_s) \right|, \\
		& \leq \frac{M}{m}\, \sup_{0\leq r \leq t} \E\left[\left| \int_s^{r} b(X_u)\d u + \int_s^{r}\sigma(X_u)\d B_u \right|\right],
	\end{align*}
	and the result follows from standard computations.
\end{proof}

\section{Mean reflected SDE  as the limit of an interacting reflected particles system}\label{sec:PMRSDE}

Having in mind the notations defined in the beginning of Section \ref{sec:EUPS} and especially equation \eqref{eq:repdeKG}, we can write the unique solution of the SDE \eqref{eq:main} as:
\begin{equation}\label{eq:main2}	
 X_t  =X_0+\int_0^t b(X_s)\d s + \int_0^t \sigma(X_s) \d B_s + \sup_{s\leq t} G_0(\mu_s),
\end{equation}
where $\mu_t$ stands for the law of 
\begin{equation*}
	U_t = X_0+\int_0^t b(X_s)\d s + \int_0^t \sigma(X_s) \d B_s.
\end{equation*}

We here are interested in the particle approximation of such a system. Our candidates are the particles:
\begin{equation*}
X_t^i =  X_0^i+ \int_0^t b(X_s^i)\d s + \int_0^t \sigma(X_s^i) \d B_s^i + \sup_{s\leq t} G_0\left(\mu_s^N\right),\quad 1\leq i \leq N,
\end{equation*}
where $B^i$ are independent Brownian motions, $X_0^i$ are independent copies of $X_0$  and $\mu_s^N$ denotes the empirical distribution at time $s$ of the particles
\begin{equation*}
U_s^i =  X_0^i + \int_0^s b(X_r^i)\d r + \int_0^s \sigma(X_r^i) \d B_r^i ,\quad 1\leq i \leq N, \qquad \mu_s^N = \dfrac{1}{N}\sum_{i=1}^N \delta_{U_s^i}.
\end{equation*}
It is worth noticing that
\begin{equation*}
	G_0\left(\mu_s^N\right) = \inf\left\{ x \geq 0 : \frac{1}{N}\sum_{i=1}^N h\left(x+U^i_s\right)\geq 0 \right\}.
\end{equation*}

\begin{remarque} Let us emphasize that the previous system of interacting particles can be seen as a multidimensional reflected SDE with oblique reflection. Indeed, if $h$ is concave, the set 
	\begin{equation*}
		C = \left\{ (x_1,\ldots,x_N)\in\R^N : h(x_1)+\ldots+h(x_N) \geq 0 \right\}
	\end{equation*}
	is convex and the system
\begin{equation*}
	\left\lbrace
	\begin{split}
	& X_t^i =  X_0^i + G_0(\mu_0^N)+ \int_0^t b(X_s^i)\d s + \int_0^t \sigma(X_s^i) \d B_s^i + K^N_t,\quad 1\leq i \leq N, \\
	& \frac{1}{N} \sum_{i=1}^N h(X_t^i) \geq 0, \quad \frac{1}{N} \sum_{i=1}^N \int_0^t h(X_t^i)\, \d K^N_s = 0,
	\end{split}
	\right.
\end{equation*}
is nothing else but the SDE driven by $b$ and $\sigma$ reflected in the convex $C$ with oblique reflexion in the direction $(1,\ldots,1)$. We refer to \cite{p._-l._lions_stochastic_1984}.
\end{remarque}

In order to prove that there is indeed a propagation of chaos effect, we introduce the following independent copies of $X$ 
\begin{equation*}
\bX_t^i = X_0^i + \int_0^t b(\bX_s^i) \d s + \int_0^t \sigma(\bX_s^i)\d B_s^i + \sup_{s\leq t} G_0(\mu_s),\quad 1\leq i\leq N,
\end{equation*}
and we couple these particles with the previous ones by choosing the same Brownian motion. 

In order to do so, we introduce the decoupled particles $\bU^i$, $1\leq i \leq N$:
\begin{equation*}
\bU_t^i = X_0^i + \int_0^t b(\bX_s^i)\d s + \int_0^t \sigma(\bX^i_s) \d B^i_s, \quad t \geq 0.
\end{equation*}
Note that for instance the particles $\bU^i$ are i.i.d. and let us denote by $\bmu^N$ the empirical measure associated to this system of particles.

We have the following result concerning the approximation \eqref{eq:main} by interacting particles system.

\begin{theoreme}\label{Prop:estiparticle}
Let assumptions \ref{H:eds} and \ref{H:h}  hold and $T>0$.
\begin{enumerate}[(i)]
\item Under assumption~\ref{H:moments}, there exists a constant $C$ depending on $b$ and $\sigma$ such that, for each $j\in\{1,\ldots,N\}$,
  \begin{equation*}
\E\left[ \sup_{s\leq T} \left|X_s^j - \bX^j_s \right|^2\right] \leq  C \exp\left( C \left(1 + \frac{M^2}{m^2}\right) (1+T^2)\right) \frac{M^2}{m^2}\, N^{-1/2}.
\end{equation*}

\item If assumption~\ref{H:smooth} is in force, then there exists a constant $C$ depending on $b$ and $\sigma$ such that, or each $j\in\{1,\ldots,N\}$,
\begin{equation*}
\E\left[ \sup_{s\leq T} \left|X_s^j - \bX^j_s \right|^2\right] \leq C \exp\left(C \left(1 + \frac{M^2}{m^2}\right) (1+T^2) \right) \frac{1+T^2}{m^2} \left(1+ \E\left[\sup_{s\leq t} |X_T|^2\right]\right) \, N^{-1}.
\end{equation*}
\end{enumerate}
\end{theoreme}

\begin{remarque}
As shown in \cite{FG15}, the rate in case (i) is optimal in full generality for the control of the Wasserstein distance of the empirical measure of  an i.i.d. sample of random variables towards its own law. It is interesting to note that the supremum over time implies no loss here, as the "propagation of chaos" is mainly herited from the reflection term through $\sup_{s\le t}W_1^2(\bar\mu^N_s,\mu_s)$.

\end{remarque}

\begin{proof}
Let $t>0$. We have, for $r\leq t$,
\begin{equation*}
\left| X_r^j - \bX^j_r \right| \leq  \int_0^r \left| b(X_s^j) - b(\bX_s^j) \right| \d s + \left|\int_0^r \left(\sigma(X_s^j) - \sigma(\bX_s^j)\right) \d B_s^j\right| {\ce +}  \left|\sup_{s\leq r} G_0(\mu_s^N)  - \sup_{s\leq r} G_0(\mu_s)\right|.
\end{equation*}
Taking into account the fact that
\begin{align*}
	\left|\sup_{s\leq r} G_0(\mu_s^N)  - \sup_{s\leq r} G_0(\mu_s)\right| & \leq \sup_{s\leq r} \left| G_0(\mu_s^N) - G_0(\mu_s)\right| \leq \sup_{s\leq t} \left| G_0(\mu_s^N) - G_0(\mu_s)\right|, \\
	& \leq \sup_{s\leq t} \left| G_0(\mu_s^N) - G_0(\bmu_s^N)\right| + \sup_{s\leq t} \left| G_0(\bmu_s^N) - G_0(\mu_s)\right|,
\end{align*}
we get the inequality
\begin{equation}
	\label{eq:u}
	\sup_{r\leq t} \left| X_r^j - \bX^j_r \right| \leq I_1 + \sup_{s\leq t} \left| G_0(\mu_s^N) - G_0(\bmu_s^N)\right| + \sup_{s\leq t} \left| G_0(\bmu_s^N) - G_0(\mu_s)\right|,
\end{equation}
where we have set
\begin{equation*}
	I_1 = \int_0^t \left| b(X_s^j) - b(\bX_s^j) \right| \d s + \sup_{r\leq t} \left|\int_0^r \left(\sigma(X_s^j) - \sigma(\bX_s^j)\right) \d B_s^j\right|.
\end{equation*}

On the one hand we have, using Doob and Cauchy-Schwartz inequalities
$$ 
\E \left[|I_1|^2\right] \leq C(1+t) \, \int_0^t \E \left[\left| X_s^j- \bX_s^j\right|^2\right]\d s,
$$
where $C$ depends only on $b$ and $\sigma$ and may change from line to line.

On the other hand, by using \eqref{regG:wass},
\begin{equation*}
	\sup_{s\leq t} \left| G_0(\mu_s^N) - G_0(\bmu_s^N)\right| \leq \frac{M}{m} \sup_{s\leq t} \frac{1}{N} \sum_{i=1}^N \left| U_s^i - \bU_s^i\right| \leq \frac{M}{m} \frac{1}{N} \sum_{i=1}^N \sup_{s\leq t} \left| U_s^i - \bU_s^i\right|,
\end{equation*}
and Cauchy-Schwartz inequality gives, since the variables are exchangeable, 
\begin{equation*}
	\E\left[\sup_{s\leq t} \left| G_0(\mu_s^N) - G_0(\bmu_s^N)\right|^2\right] \leq \frac{M^2}{m^2} \frac{1}{N} \sum_{i=1}^N \E\left[\sup_{s\leq t} \left| U_s^i - \bU_s^i\right|^2\right] = \frac{M^2}{m^2} \E\left[\sup_{s\leq t} \left| U_s^j - \bU_s^j\right|^2\right].
\end{equation*} 
Since,
\begin{equation*}
	U_s^j - \bU_s^j = \int_0^s \left( b(X_r^j) - b(\bX_r^j) \right) \d r + \int_0^s \left(\sigma(X_r^j) - \sigma(\bX_r^j)\right) \d B_r^j
\end{equation*}
the same computations as we did before lead to
\begin{equation*}
	\E\left[\sup_{s\leq t} \left| G_0(\mu_s^N) - G_0(\bmu_s^N)\right|^2\right] \leq C \frac{M^2}{m^2}(1+t) \, \int_0^t \E \left[\left| X_s^j- \bX_s^j\right|^2\right]\d s.
\end{equation*}
Hence, with the previous estimates we get, coming back to~\eqref{eq:u},
\begin{align*}
	\E\left[\sup_{r\leq t} \left| X_r^j - \bX^j_r \right|^2\right] & \leq K \, \int_0^t \E \left[\left| X_s^j- \bX_s^j\right|^2\right]\d s + 3\, \E\left[ \sup_{s\leq t} \left| G_0(\bmu_s^N) - G_0(\mu_s)\right|^2\right], \\
	 & \leq K \, \int_0^t \E \left[ \sup_{r\leq s}\left| X_s^j- \bX_s^j\right|^2\right]\d s + 3\, \E\left[ \sup_{s\leq t} \left| G_0(\bmu_s^N) - G_0(\mu_s)\right|^2\right],
\end{align*}
where $K= C(1+t) \left(1+M^2/m^2\right)$. Thanks to Gronwall's Lemma
\begin{equation*}
	\E\left[\sup_{r\leq t} \left| X_r^j - \bX^j_r \right|^2\right] \leq C e^{Kt} \, \E\left[ \sup_{s\leq t} \left| G_0(\bmu_s^N) - G_0(\mu_s)\right|^2\right].
\end{equation*}

By Lemma \ref{lemme:regG} we know that 
\begin{equation*}
	\E\left[ \sup_{s\leq t} \left| G_0(\bmu_s^N) - G_0(\mu_s)\right|^2\right] \leq \frac{1}{m^2}\E\left[\sup_{s\leq t} \left| \int h(\bar{G}_0(\mu_s)+\cdot) (\d \bmu_s^N - \d \mu_s)\right|^2\right],
\end{equation*}
from which we deduce that
\begin{equation}\label{eq:rateconvgen}
	\E\left[\sup_{r\leq t} \left| X_r^j - \bX^j_r \right|^2\right] \leq C e^{Kt} \, \frac{1}{m^2}\E\left[\sup_{s\leq t} \left| \int h(\bar{G}_0(\mu_s)+\cdot) (\d \bmu_s^N - \d \mu_s)\right|^2\right].
\end{equation}

\smallskip

\noindent\emph{Proof of (i).}
Since the function $h$ is, at least, a Lipschitz function, we understand that the rate of convergence follows from the convergence of empirical measure of i.i.d. diffusion processes. The crucial point here is that we consider a uniform (in time) convergence, which may possibly damage the usual rate of convergence. We will see that however here, we are able to conserve this optimal rate. Indeed, in full generality (\emph{i.e.} if we only suppose that \ref{H:h} holds) we get that:
\begin{equation*}
	\frac{1}{m^2}\E\left[\sup_{s\leq t} \left| \int h(\bar{G}_0(\mu_s)+\cdot) (\d \bmu_s^N - \d \mu_s)\right|^2\right]  \leq \frac{M^2}{m^2} \E \left[\sup_{s\leq t} W_1^2\left(\bmu_s^N,\mu_s\right)\right], 
\end{equation*}
Thanks to the additional assumption~\ref{H:moments} and to Proposition~\ref{prop:polymom}, we will adapt and simplify the proof of Theorem 10.2.7 of \cite{RR98} using recent results about the control Wasserstein distance of empirical measures of i.i.d. sample to the true law by \cite{FG15}, to obtain
\begin{equation*}
	\E \left[\sup_{s\leq 1} W_1^2\left(\bmu_s^N,\mu_s\right)\right] \leq C\, N^{-1/2}.
\end{equation*}
Indeed, let $n$ be a positive integer and set $t_k=k/n$, $0\le k\le n$. As in \cite{RR98}, denote
$$Z_k=\sup_{t_k\le t \le t_{k+1}} W_1^2(\bar \mu^N_s,\mu_s).$$
Then
$$\sup_{s\leq 1} W_1^2\left(\bmu_s^N,\mu_s\right)\le 3\left[\max_k Z_k+\max_k W_1^2(\bar\mu^N_{t_k},\mu_{t_k})+\max_k\sup_{t_k\le t \le t_{k+1}} W_1^2(\mu_{t_k},\mu_t)\right].$$
Now, using the regularity properties of Proposition~\ref{prop:polymom} and proceeding exactly as in \cite[Th. 10.2.7]{RR98} we have that there exists $C>0$ such that
$$W_1^2(\mu_{t_k},\mu_t)\le \frac Cn,\qquad\qquad \E\max Z_k\le \frac{C}{\sqrt{n}}.$$
We are left to control $\E\max_k W_1^2(\bar\mu^N_{t_k},\mu_{t_k})$: first remark that
$$\E\max_k W_1^2(\bar\mu^N_{t_k},\mu_{t_k})\le \sqrt{n}\sqrt{\max_k\E W_1^4(\bar\mu^N_{t_k},\mu_{t_k})}.$$
Use now assumption~\ref{H:moments} and Theorem 2 (case (3)) in \cite{FG15} to get 
$$\E W_1^4(\bar\mu^N_{t_k},\mu_{t_k})\le \frac C{N^2}$$
from which we deduce that 
$$\E \left[\sup_{s\leq 1} W_1^2\left(\bmu_s^N,\mu_s\right)\right]\le C\left[\frac {\sqrt n}{N}+\frac1{\sqrt{n}}\right]$$
and optimization procedure on $n$ finishes the proof.
Let us emphasize that this result does not care of the fact that $\bmu^N$ is an empirical measure associated to i.i.d. copies of a {\it diffusion process}. 

\noindent\emph{Proof of (ii).}
In the case where $h$ is a twice continuously differentiable function with bounded derivatives (\emph{i.e.} under \ref{H:smooth}), we succeed to take benefit from the fact that $\bmu^N$ is an empirical measure associated to i.i.d. copies of diffusion process, in particular we can get rid of the supremum in time. In view of \eqref{eq:rateconvgen}, we need a sharp estimate of 
\begin{equation*}
	\E\left[ \sup_{s\leq t} \left| \int h\left(\bar{G}_0(\mu_s)+\cdot\right) (\d \bmu_s^N - \d \mu_s)\right|^2 \right].
\end{equation*}

Let us first observe that $s\longmapsto \bar{G}_0(\mu_s)$ is locally Lipschitz continuous. Indeed, since by definition $H\left(\bar{G}_0(\mu_t), \mu_t\right)=0$, if $s<t$, using~\eqref{reg:Hspace},
\begin{align*}
	\left| \bar{G}_0(\mu_s) - \bar{G}_0(\mu_t)\right | & \leq \frac{1}{m} \left| H\left(\bar{G}_0(\mu_s), \mu_t\right)\right| = \frac{1}{m} \left| \E\left[h\left(\bar{G}_0(\mu_s)+U_t\right)\right]\right|, \\
	& = \frac{1}{m}\left| \E\left[h\left(\bar{G}_0(\mu_s)+U_s + \int_s^t b(X_r) dr + \int_s^t \sigma(X_r)dB_r\right)\right] \right|.
\end{align*}
We get from Itô's formula, setting
\begin{equation*}
	\bar{\mathcal{L}}_{y}:= b(y)\frac{\p}{\p x} + \sigma\sigma^*(y)\frac{\p^2}{\p x^2},
\end{equation*}
\begin{align*}
	\E\left[h\left(\bar{G}_0(\mu_s)+U_t\right)\right] & = \E\left[ h\left(\bar{G}_0(\mu_s)+U_s\right)\right] + \int_s^t \E\left[\bar{\mathcal{L}}_{X_r} h\left(\bar{G}_0(\mu_s)+U_r\right)\right] dr, \\	
	& = H\left(\bar{G}_0(\mu_s),\mu_s\right) + \int_s^t \E\left[\bar{\mathcal{L}}_{X_r} h\left(\bar{G}_0(\mu_s)+U_r\right)\right] dr, \\
	& = \int_s^t \E\left[\bar{\mathcal{L}}_{X_r} h\left(\bar{G}_0(\mu_s)+U_r\right)\right] dr.
\end{align*}
Since $h$ has bounded derivatives and $\sup_{s\leq T} |X_s|$ is a square integrable random variable for each $T>0$ (see Corollary \ref{prop:polymom}), the result follows easily.

Let us denote by $\psi$ the Radon-Nikodym derivative of $\bar{G}_0(\mu_\cdot)$. By definition, we have, denoting by $V^i$ the semimartingale $s\longmapsto \bar{G}_0(\mu_s)+ \bU^i_s$, since $\bU^i$ are independent copies of $U$,
\begin{align*}
	R_N(s) := \int h\left(\bar{G}_0(\mu_s)+\cdot \right) (\d \bmu_s^N - \d \mu_s) & = \frac{1}{N}\sum_{i=1}^N h\left( \bar{G}_0(\mu_s)+ \bU^i_s\right) - \E\left[h\left( \bar{G}_0(\mu_s) + U_s\right)\right], \\
	& = \frac{1}{N}\sum_{i=1}^N \left\{ h\left( \bar{G}_0(\mu_s)+ \bU^i_s\right) -\E\left[h\left( \bar{G}_0(\mu_s)+ \bU^i_s\right)\right]\right\}, \\
	&= \frac{1}{N}\sum_{i=1}^N \left\{ h\left( V^i_s\right) -\E\left[h\left( V^i_s\right)\right]\right\},
\end{align*}
It follows from Itô's formula 
\begin{align*}
	h\left( V^i_s\right) & = h\left( V^i_0\right) + \int_0^s h'\left(V^i_r\right) \psi_r \d r + \int_0^s \bar{\mathcal{L}}_{\bX_r}h\left(V_r^i \right)\d r + \int_0^s h'\left(V_r^i\right)\sigma(\bX_r^i) \d B_r^i, \\
	& = h\left( V^i_0\right) + \int_0^s \left\{ h'\left(V^i_r\right) \psi_r + \bar{\mathcal{L}}_{\bX_r}h\left(V_r^i \right)\right\} \d r + \int_0^s h'\left(V_r^i\right)\sigma(\bX_r^i) \d B_r^i,
\end{align*}
together with
\begin{align*}
	\E\left[h\left( V^i_s\right)\right] & = \E\left[h\left( V^i_0\right)\right] + \int_0^s \E\left[ h'\left(V^i_r\right) \psi_r + \bar{\mathcal{L}}_{\bX_r}h\left(V_r^i \right)\right] \d r, \\
	& = 0 + \int_0^s \E\left[ h'\left(V^i_r\right) \psi_r + \bar{\mathcal{L}}_{\bX_r}h\left(V_r^i \right)\right] \d r.
\end{align*}

We deduce immediately that
\begin{align*}
	R_N(s) & = \frac{1}{N}\sum_{i=1}^N h\left( V^i_0\right) +  \frac{1}{N}\sum_{i=1}^N  \int_0^s C^i_r \,\d r + M_N(s), \\
	& = \frac{1}{N}\sum_{i=1}^N h\left( V^i_0\right) + \int_0^s \left(\frac{1}{N}\sum_{i=1}^N  C^i(r)\right) \d r + M_N(s)
	\intertext{where we have set}
	C^i({\ce r}) & = h'\left(V^i_r\right) \psi_r + \bar{\mathcal{L}}_{\bX_r}h\left(V_r^i \right) - \E\left[ h'\left(V^i_r\right) \psi_r + \bar{\mathcal{L}}_{\bX_r}h\left(V_r^i \right)\right], \\
	M_N(s) & = \frac{1}{N}\sum_{i=1}^N \int_0^s h'\left(V_r^i\right)\sigma(\bX_r^i) \d B_r^i. 
\end{align*}
As a byproduct,
\begin{align*}
	\sup_{s\leq t} \left| R_N(s) \right| & \leq \left| \frac{1}{N}\sum_{i=1}^N h\left( V^i_0\right) \right| + \sup_{s\leq t} \int_0^s \left| \frac{1}{N}\sum_{i=1}^N  C^i(r) \right| \d r + \sup_{s\leq t} |M_N(s)|, \\
	& \leq \left| \frac{1}{N}\sum_{i=1}^N h\left( V^i_0\right) \right| + \int_0^t \left| \frac{1}{N}\sum_{i=1}^N  C^i(r) \right| \d r+ \sup_{s\leq t} |M_N(s)|.
\end{align*}
We get, using Cauchy-Schwartz inequality, since $U^i$ and $\bar X^i$ are i.i.d,
\begin{multline*}
	\E\left[\sup_{s\leq t} |R_N(s)|^2\right] \\
	 \leq 3 \left\{ \mathrm{Var}\left[\frac{1}{N}\sum_{i=1}^N h\left( V^i_0\right)\right]  + \E\left[\left(\int_0^t \left| \frac{1}{N}\sum_{i=1}^N  C^i(r) \right| \d r\right)^2 \right] + \E\left[ \sup_{s\leq t} | M_N(s)|^2\right]\right\}, \\
	 \leq 3 \left\{ \mathrm{Var}\left[\frac{1}{N}\sum_{i=1}^N h\left( V^i_0\right)\right]  + t\, \E\left[\int_0^t \left| \frac{1}{N}\sum_{i=1}^N  C^i(r) \right|^2 \d r\right] + \E\left[ \sup_{s\leq t} | M_N(s)|^2\right]\right\}, \\
	 = 3 \left\{ \mathrm{Var}\left[\frac{1}{N}\sum_{i=1}^N h\left( V^i_0\right)\right]  + t\,\int_0^t \mathrm{Var}\left( \frac{1}{N}\sum_{i=1}^N  C^i(r) \right) \d r + \E\left[ \sup_{s\leq t} | M_N(s)|^2\right] \right\}.
\end{multline*}
Thus, we get
\begin{multline*}
	\E\left[\sup_{s\leq t} |R_N(s)|^2\right] \\
	\leq \frac{3}{N} \mathrm{Var}\left[h\left(V_0\right)\right] + \frac{3t}{N}\, \int_0^t \mathrm{Var}\left( C(r)\right) \d r + 3\, \E\left[ \sup_{s\leq t} | M_N(s)|^2\right], \\
	 = \frac{3}{N} \mathrm{Var}\left[h\left(V_0\right)\right] + \frac{3t}{N}\, \int_0^t \mathrm{Var}\left(h'\left(V_r\right) \psi_r + \bar{\mathcal{L}}_{X_r}h\left(V_r \right)\right) \d r + 3\, \E\left[ \sup_{s\leq t} | M_N(s)|^2\right].
\end{multline*}
Since $M_N$ is a martingale with
\begin{equation*}
	\left\langle M_N \right\rangle_t = \frac{1}{N^2} \sum_{i=1}^N \int_0^t  \left( h'\left(V^i_r\right)\sigma\left(\bX^i_r\right) \right)^2  \d r,
\end{equation*}
Doob's inequality leads to
\begin{align*}
	\E\left[ \sup_{s\leq t} | M_N(s)|^2\right] & \leq 4\, \E\left[|M_N(t)|^2\right], \\
	&= \frac{4}{N^2}\, \sum_{i=1}^N \int_0^t \E\left[\left( h'(V^i_r)\sigma(\bX^i_r) \right)^2\right] \d r, \\
	&= \frac{4}{N}\, \int_0^t \E\left[\left( h'(V_r)\sigma(X_r) \right)^2\right] \d r.
\end{align*}
Finally, using the fact that $h$ has bounded derivatives, $b$ and $\sigma$ are Lipschitz, we get
\begin{equation*}
	\E\left[ \sup_{s\leq t} | R_N(s)|^2\right]  \leq C(1+t^2) \left(1+ \E\left[\sup_{s\leq t} |X_s|^2\right]\right) \, N^{-1}.
\end{equation*}
This gives the result coming back to~\eqref{eq:rateconvgen}.

\end{proof}

\section{A numerical scheme for MRSDE}
\label{sec:NSMRSDE}
We are interested in the  numerical approximation of the SDE \eqref{eq:main} on $[0,T]$. Here are the main steps of the scheme. Let $0=T_0 < T_1 <\cdots<T_n=T$ be a subdivision of $[0,T]$. Given this subdivision, we denote by ``$\underbar{\ }$'' the mapping $s \mapsto \s= T_{k}$ if $s \in [T_k,T_{k+1})$, $k\in\{0,\cdots ,n-1\}$. For simplicity, we consider only the case of regular subdivisions: for a given integer $n$, $T_k = k \, T/n$, $k=0,\ldots,n$.

Let us recall that we proved in the previous section that particles system
\begin{align*}
X_t^i & =  X_0^i+ \int_0^t b(X_s^i)\d s + \int_0^t \sigma(X_s^i) \d B_s^i + \sup_{s\leq t} G_0\left(\mu_s^N\right),\quad 1\leq i \leq N, \\
\intertext{where we have set}
U_s^i & =  X_0^i + \int_0^s b(X_r^i)\d r + \int_0^s \sigma(X_r^i) \d B_r^i ,\quad 1\leq i \leq N, \qquad \mu_s^N = \dfrac{1}{N}\sum_{i=1}^N \delta_{U_s^i}.
\end{align*}
$B^i$ being independent Brownian motions and $X_0^i$ being independent copies of $X_0$, converges toward the solution to~\eqref{eq:main}. Thus, the numerical approximation is obtained by an Euler scheme applied to this particles system. We introduce the following discrete version of the particles system
\begin{align*}
\tilde{X}_t^i & =  X_0^i+ \int_0^t b\left(\tX_{\s}^i\right)\d s + \int_0^t \sigma\left(\tX_{\s}^i\right) \d B_s^i + \sup_{s\leq t} G_0\left(\tmu_{\s}^N\right),\quad 1\leq i \leq N, \\
\intertext{with the notation}
\tU_t^i & =  X_0^i + \int_0^t b\left(\tX_{\s}^i\right)\d s + \int_0^t \sigma\left(\tX_{\s}^i\right) \d B_s^i ,\quad 1\leq i \leq N, \qquad \tmu_s^N = \dfrac{1}{N}\sum_{i=1}^N \delta_{\tU_s^i}.
\end{align*}


\subsection{Scheme}
Using the notations given above, the result on the interacting system of mean reflected particles of the MR-SDE of Section \ref{sec:PMRSDE} and Remark \ref{rem:markovK}, we deduce the following algorithm for the numerical approximation of the MR-SDE: 

\begin{algorithm}[h!]
\caption{Particle approximation}\label{algo:KF}
\begin{algorithmic}[1]
\For{$1\leq j \leq N$} 
\State $\l(\l(\tX^{\tmu^N}_0\r)^j,\l(\tU^{\tmu^N}_0\r)^j,\hmu_0^N\r) =(x,x,\delta_x)$
\EndFor
\For{$1 \leq k \leq n$}
\For{$1\leq j \leq N$}
\State $G^j \sim \mathcal{N}(0,1)$
\State $\l(\tU_{T_k}^{\tmu^N}\r)^j = \l(\tU_{T_{k-1}}^{\tmu^N}\r) +   (T/n)b\l(\l(\tX_{T_{k-1}}^{\tmu^N}\r)^j\r) +  \sqrt{(T/n)} \sigma\l(\l(\tX_{T_{k-1}}^{\tmu^N}\r)^j\r) G^j$
\EndFor
\State $\tmu_{T_k}^N = N^{-1}\sum_{j=1}^N\delta_{\l(\tU_{T_k}^{\tmu^N}\r)^j}$
\State  $\Delta_{k}\hat{K}^N = \sup_{l\leq k} G_0(\tmu_{T_l}^N) - \sup_{l\leq k-1} G_0(\tmu_{T_l}^N)$
\For{$1\leq j \leq N$}
\State $ \l(\tX_{T_k}^{\tmu^N}\r)^j = \l(\tX_{T_{k-1}}^{\tmu^N}\r)^j +(T/n)  b\l(\l(\tX_{T_{k-1}}^{\tmu^N}\r)^j\r) +  \sqrt{(T/n)} \sigma\l(\l(\tX_{T_{k-1}}^{\tmu^N}\r)^j\r) G^j+ \Delta_{k}\hat{K}^N$
\EndFor
\EndFor
\end{algorithmic}
\end{algorithm}

\begin{remarque}
We emphasize that, at each step $k$ of the algorithm, we approximate the increment of the reflection process $K$ by the increment of the this approximation: 
$$\Delta_{k}\hat{K}^N : = \sup_{l\leq k} G_0(\tmu_{T_l}^N) - \sup_{l\leq k-1} G_0(\tmu_{T_l}^N).$$ 
As suggested in Remark \ref{rem:markovK}, this increment can be approached by:
\begin{multline*}
	\widehat{\Delta_{k}K}^N  : = \\
	\inf\left\{x \geq 0 : \frac{1}{N} \sum_{i=1}^N h\left(x +  \l(\tX_{T_{k-1}}^{\tmu^N}\r)^i + \frac{T}{n}  b\l(\l(\tX_{T_{k-1}}^{\tmu^N}\r)^i\r) +  \frac{\sqrt T}{\sqrt n} \sigma\l(\l(\tX_{T_{k-1}}^{\tmu^N}\r)^i\r) G^j\right) \geq 0 \right\}.
\end{multline*}
Indeed, using the same kind of arguments as in the sketch of the proof of Theorem  \ref{th:wp}, one can show that the increments of the approximated reflection process are equals to the approximation of the increments:
$$\forall k \in \{1,\ldots,n\}:\quad \widehat{\Delta_{k}K}^N =\Delta_{k}\hat{K}^N.$$
\end{remarque}

\subsection{Scheme error}

\begin{proposition}\label{en:euler}
	Let $T>0$, $N$ and $n$ be two non-negative integers and let assumptions \ref{H:eds}, \ref{H:h} and \ref{H:moments} hold. There exists a constant $C$, depending on $T$, $b$, $\sigma$, $h$ and $X_0$ but independent of $N$, such that: for all $i=1,\ldots,N$
	\begin{equation*}
		\E\left[\sup_{s\leq T} \left|X^i_t - \tX^i_t\right|^2\right] \leq C \, \frac{\log n}{n}.
	\end{equation*}
\end{proposition}
Let us admit for the moment the following result that will be useful for our proof.
\begin{lemme}\label{en:BMD}
	There exists a constant $C$ such that
	\begin{equation*}
		\E\left[\sup_{s\leq T} |B_s-B_{\s}|^4\right] \leq C \left(\frac{\log n}{n}\right)^2.
	\end{equation*}
\end{lemme}
We may now proceed to the proof of proposition \ref{en:euler}
\begin{proof}

Let us fix $i\in\{1,\ldots,N\}$ and $T>0$. We have, for $t\leq T$, 
\begin{align*}
\left|X_{t}^{i} - \tX_{t}^{i}\right| &\leq  \int_{0}^{t} \left| b\left(X_s^{i}\right) - b\left(\tX_{\s}^{i}\right)\right| \d s + \left\vert \int_0^{t} \left( \sigma\left(X_s^{i}\right) - \sigma\left(\tX_{\s}^{i}\right)\right) \d B_s \right\vert \\
&\quad + \sup_{s \leq t}  \left| G_0\left(\mu^N_s\right)- G_0\left(\tmu^N_{\s}\right) \right|.
\end{align*}
Hence, using Cauchy-Schwartz and Doob inequality we get:
\begin{eqnarray}\label{eq:estierror1}
\E\left[\sup_{s\leq t}\left|X_{s}^{i}- \tX_{s}^{i}\right|^2\right] \leq  C\, \int_{0}^{t} \E\left[\left|X_{s}^{i} - \tX_{\s}^{i}\right|^2\right] \d s + \E\left[\sup_{s\leq t} \left|G\left(\mu^N_s\right)-G\left(\tmu^N_{\s}\right)\right|^2\right].
\end{eqnarray}

We now deal with the last term in the right hand side: from Lemma \ref{lemme:regG} we have
\begin{align*}
\E\left[\sup_{s\leq t} \left|G\left(\mu^N_s\right)-G\left(\tmu^N_{\s}\right)\right|^2\right] & \leq \left(\frac{M}{m}\right)^2 \E\left[ \sup_{s \leq t}\frac{1}{N} \sum_{j=1}^N \left| U^{j}_s - \tU^{j}_s \right|^2 \right], \\
& \leq  2\left(\frac{M}{m}\right)^2 \E\left[ \sup_{s \leq t}\frac{1}{N} \sum_{j=1}^N \Bigg\{\left| U^{j}_s - \tU^{j}_s \right|^2 + \left| \tU^{j}_s - \tU^{j}_{\s} \right|^2\Bigg\} \right], 
\intertext{and, since the law of the particles is independent by permutations,}
& \leq 2\left(\frac{M}{m}\right)^2 \E\left[ \sup_{s \leq t}\left| U^{i}_s - \tU^{i}_s \right|^2 + \sup_{s \leq t} \left| \tU^{i}_s - \tU^{i}_{\s} \right|^2 \right].
\end{align*}

For the first term of the right hand side, let us observe that, since the law of the particles is independent by permutations,
\begin{multline*}
\E\left[ \sup_{s \leq t}\left| U^{i}_s - \tU^{i}_s \right|^2\right]  \\
\leq   \E \left[\sup_{s \leq t}\Bigg\{ \left(\int_0^{s} \left| b\left(X_r^{i}\right) - b\left(\tX_{\ur}^{i}\right) \right |\d r\right)^2 
+ \left| \int_0^s  \left(\sigma\left(X_r^{i}\right) - \sigma\left(\tX_{\ur}^{i}\right)\right) \d B_r \right |^2 \Bigg\}\right],\\
\leq     t ||b||_{\Lip}^2 \int_0^{t} \E\left[\left| X_s^{i}- \tX_{\s}^{i}\right |^2\right]\d s 
+ 2 ||\sigma ||_{\Lip}^2\int_0^{t} \E \left[\left| X_s^{i} - \tX_{\s}^{i} \right |^2\right]\d s .
\end{multline*}

We have for the second term 
\begin{equation*}
\E\left[\sup_{s\leq t}\left| \tU^{i}_s - \tU^{i}_{\s} \right|^2 \right] \leq \E\left[\sup_{s\leq t}\Bigg\{ \l|b(\tX^{i}_{\s})\r|^2(s-\s)^2 + \l|\sigma(\tX^{i}_{\s})\r|^2 (B_s-B_{\s})^2\Bigg\}\right],
\end{equation*}
and obviously
\begin{equation*}
	\E\left[\sup_{s\leq t} \l|b(\tX^{i}_{\s})\r|^2(s-\s)^2 \right] \leq C\left(1+\E\left[\sup_{s\leq T} \l|\tX_s^{i}\r|^2\right]\right) \l(\frac{T}{n}\r)^2.
\end{equation*}
On the other hand, 
\begin{align*}
\E\left[\sup_{s\leq t} \l|\sigma(\tX^{i}_{\s})\r|^2 (B_s-B_{\s})^2 \right] & \leq \E\left[\sup_{s\leq T} \l|\sigma(\tX^{i}_{\s})\r|^4\right]^{1/2} \, \E\left[\sup_{s\leq T} |B_s-B_{\s}|^4\right]^{1/2}, \\
& \leq C\left(1+\E\left[\sup_{s\leq T} \l|\tX_s^{i}\r|^4\right]^{1/2}\right)\, \E\left[\sup_{s\leq T} |B_s-B_{\s}|^4\right]^{1/2}.
\end{align*}

We have 
\begin{equation}\label{eq:deltaU}
	\E\left[\sup_{s\leq t}\left| \tU^{i}_s - \tU^{i}_{\s} \right|^2 \right] \leq C\, \frac{\log n}{n},
\end{equation}
from which we derive the inequality
\begin{equation}\label{eq:estierror2}
\E\left[\sup_{s\leq t} \left|G\left(\mu^N_s\right)-G\left(\tmu^N_{\s}\right)\right|^2\right] \leq  C \Bigg\{ \l(\frac{\log n}{n}\r) +  \int_0^{t} \E\left[\left| X_s^{i}- \tX_{\s}^{i}\right |^2\right]\d s \Bigg\},
\end{equation}
and taking into account~\eqref{eq:estierror1} we get
\begin{equation}\label{eq:estierror3}
	\E\left[\left|X_{t}^{i}- \tX_{t}^{i}\right|^2\right] \leq C \Bigg\{ \l(\frac{\log n}{n}\r) +  \int_0^{t} \E\left[\left| X_s^{i}- \tX_{\s}^{i}\right |^2\right]\d s \Bigg\}.
\end{equation}

Since
\begin{align*}
  \E\left[\left| X_s^{i} - \tX_{\s}^{i} \right |^2\right] & \leq 2\,\E\left[\left|X_s^{i} -\tX_{s}^{i}\right |^2\right] + 2\,\E\left[\left|\tX_{s}^{i} -\tX_{\s}^{i}\right |^2\right] , \\
  & = 2\,\E\left[\left|X_s^{i} -\tX_{s}^{i}\right |^2\right] + 2\,\E\left[\left|\tU_{s}^{i} -\tU_{\s}^{i}\right |^2\right],
\end{align*}
it follows from \eqref{eq:deltaU} and \eqref{eq:estierror3} that
\begin{equation*}
	\E\left[\left|X_{t}^{i}- \tX_{t}^{i}\right|^2\right] \leq C \left\{ \l(\frac{\log n}{n}\r) +  \int_0^{t} \E\left[\left| X_s^{i}- \tX_{s}^{i}\right |^2\right]\d s \right\},
\end{equation*}
and we conclude the proof with Gronwall lemma.
\end{proof}

\begin{proof}[Proof of Lemma~\ref{en:BMD}]
	Let us start by observing that
	\begin{align*}
		\sup_{s\leq T} |B_s-B_{\s}| & = \max_{k=0,\ldots,n-1} \sup_{T_k\leq s\leq T_{k+1}} |B_s-B_{T_k}|, \\
		&= \max_{k=0,\ldots,n-1} \sup_{T_k\leq s\leq T_{k+1}} \max\left(B_s-B_{T_k}, -(B_s-B_{T_k})\right), \\
		& \leq \max_{k=0,\ldots,n-1} \sup_{T_k\leq s\leq T_{k+1}} \left(B_s-B_{T_k}\right) + \max_{k=0,\ldots,n-1} \sup_{T_k\leq s\leq T_{k+1}} \left(-\left(B_s-B_{T_k}\right)\right).
	\end{align*}
	Since the random variables $\sup_{T_k\leq s\leq T_{k+1}} \left(B_s-B_{T_k}\right)$, $k=0,\ldots, n-1$, as well as the variables $\sup_{T_k\leq s\leq T_{k+1}} \left(-\left(B_s-B_{T_k}\right)\right)$, $k=0,\ldots, n-1$, are independent and have the same law as $|B_{T_1}|$,
	\begin{equation*}
		\E\left[\sup_{s\leq T}|B_s-B_{\s}|^4\right] \leq 8 \frac{T^2}{n^2}\, \E\left[\max\left(|X_1|^4,\ldots,|X_n|^4\right)\right],
	\end{equation*}
	where $X_k$, $k=1,\ldots,n$ are independent normal gaussian random variables.
	
	Let $f$ be the function defined on $\R_+$ by $f(x)= e^{\sqrt{1+x/16}}$; $f$ is convex, increasing with values in $[e,+\infty[$. The inverse of $f$ is concave on $[e,+\infty[$, increasing and $f^{-1}(y) = 16\left[(\log y)^2-1\right]$. We have, by Jensen inequality,
	\begin{equation*}
		\E\left[\max_{k=1,\ldots,n} |X_k|^4 \right]  = \E\left[\max_{k=1,\ldots,n}f^{-1}\circ f\left(|X_k|^4\right)\right] \leq f^{-1}\left(\E\left[\max_{k=1,\ldots,n} f\left(|X_k|^4\right)\right]\right),
	\end{equation*}
	from which we deduce that
	\begin{equation*}
		\E\left[\max_{k=1,\ldots,n} |X_k|^4 \right] \leq f^{-1}\left(\E\left[\sum_{k=1}^n f\left(|X_k|^4\right)\right]\right) = f^{-1}\left(n\E\left[ f\left(|X_1|^4\right)\right]\right)\leq f^{-1}\left(ne\E\left[ e^{|X_1|^2/4}\right]\right).
	\end{equation*}
	Finally, we have
	\begin{align*}
		 f^{-1}\left(ne\E\left[ e^{|X_1|^2/4}\right]\right) = f^{-1}\left(ne\sqrt{2}\right)= 16\left(\left(1+\log n + \log(2)/2\right)^2-1\right).
	\end{align*}
	This concludes the proof of the lemma.
\end{proof}

Let us recall that $\bX^i$ are independent and identically distributed copies of $X$:
\begin{equation*}
\bX_t^i = X_0^i + \int_0^t b(\bX_s^i) \d s + \int_0^t \sigma(\bX_s^i)\d B_s^i + \sup_{s\leq t} G_0(\mu_s),\quad 1\leq i\leq N,
\end{equation*}
where $\mu_s$ stands for the law of
\begin{equation*}
	U_s = X_0 +\int_0^s b(X_r)\, dr + \int_0^s \sigma(X_r)\,dB_r.
\end{equation*}

\begin{theoreme}
	\label{thm:NS} 
	Let $T>0$, $N$ and $n$ be two non-negative integers.  Let assumptions \ref{H:eds}, \ref{H:h} and \ref{H:moments} hold.
\begin{enumerate}[(i)]
	\item There exists a constant $C$ depending on $T$, $b$, $\sigma$, $h$ and $X_0$ such that: for all $i\in\{1,\ldots,N\}$,
	\begin{equation*}
		\E\left[\sup_{t\leq T}\l|\bX^i_t -\tX_t^{i}\r|^2\right] \leq C\l( \frac{\log n}{n} + N^{-1/2} \r).
	\end{equation*}
\item If in addition \ref{H:smooth} hold, there exists a positive constant $C$ depending on $T$, $b$, $\sigma$, $h$ and $X_0$ such that: for all $i\in\{1,\ldots,N\}$,
\begin{equation*}
	\E\left[\sup_{t\leq T}\l|\bX^i_t -\tX_t^{i}\r|^2\right] \leq C\l( \frac{\log n}{n} + N^{-1} \r).
\end{equation*}
\end{enumerate}
\end{theoreme}

\begin{proof}
	The proof is straightforward writing
	\begin{equation*}
		\l|\bX^i_t -\tX_t^{i}\r| \leq \l|\bX^i_t -X_t^{i}\r| + \l|X^i_t -\tX_t^{i}\r|
	\end{equation*}
	and using Theorem~\ref{Prop:estiparticle} and Proposition~\ref{en:euler}.
\end{proof}

\section{Numerical illustrations}
\label{sec:NI}
Throughout this section, we consider, on $[0,T]$ the following sort of processes:
\begin{equation}\label{eq:mainNumIll}	
	\left\lbrace\begin{array}{ll}
	 \displaystyle X_t  =X_0-\int_0^t \left(\beta_s + a_s X_s\right)\d s + \int_0^t \left(\sigma_s + \gamma_s X_s\right) \d B_s + K_t, \\
	 \displaystyle  \E[h(X_t)] \geq 0, \quad \int_0^t \E[h(X_s)] \, \d K_s = 0,
	\end{array}\right.
	\end{equation}
where $(\beta_t)_{t\geq 0}$, $(a_t)_{t \geq 0}$, $(\sigma_t)_{t \geq 0}$ and $(\gamma_t)_{t \geq 0}$ are bounded adapted processes. This sort of processes allow us to make some explicit computations leading us to illustrate the algorithm. Our results are presented for different diffusion and functions $h$ that are summarized below. 

\subsection*{Linear constraint}
We first consider cases where $h: \R \ni x \longmapsto x-p \in \R$. 
\begin{enumerate}[\text{Case} (i)]
\item \label{item:firstcaseex1} Drifted Brownian motion: $\beta_t = \beta>0$, $a_t = \gamma_t = 0$, $\sigma_t=\sigma>0$, $X_0=x_0\geq p$. We have 
$$ 
K_t =  \left(p + \beta t -x_0 \right)^+.
$$
\item \label{item:firstcaseex2} Ornstein Uhlenbeck process: $\beta_t = \beta>0$, $a_t = a>0$, $\gamma_t=0$, $\sigma_t=\sigma >0$, $ X_0 = x_0$  with $x_0\geq p > -\beta/a$. We have 
\begin{eqnarray*}
K_t = (ap+\beta)(t-t^\star)1_{t \geq t^{\star}},\text{ where }t^\star = \frac{1}{a} \left(\ln(x_0+\beta/a) - \ln(p+\beta/a)\right).
\end{eqnarray*}

\item \label{item:firstcaseex3}  Ornstein Uhlenbeck process with stochastic mean parameter: $\beta_t = \beta>0$, $a_t = -\epsilon B_t$, $\epsilon>0$, $\gamma_t=0$, $\sigma_t=\sigma >0$, $X_0 = x_0$, $x_0>p$. When $\epsilon\to 0^{+}$ 
\begin{equation*}
K_t = \left( p -x_0 +\beta t - \sigma \epsilon  \frac{t^2}{2} \right) 1_{[t^\star, \bar{t}[}(t) + \left( -(x_0-p) + \frac{\beta^2}{2\epsilon\sigma} \right) 1_{t\geq \bar t}+ o(\epsilon),
\end{equation*}
 where  $\bar{t}=\beta/(\epsilon\sigma)$ and $t^\star = \left(\beta - \sqrt{\beta^2 - 2(x_0-p)\sigma \epsilon}\right)/(\epsilon \sigma)$.

\item \label{item:firstcaseex4} Black and Sholes process: $\beta_t = \beta >0$, $a_t = a>0$, $\sigma_t = 0$, $\gamma_t = \gamma>0$. Then
$$
K_t = (ap+\beta)(t-t^\star)1_{t \geq t^{\star}},\text{ where }t^\star = \frac{1}{a} \left(\ln(x_0+\beta/a) - \ln(p+\beta/a)\right).$$
\end{enumerate}

\subsection*{Nonlinear constraint}
 Secondly, we illustrate the case of non-linear function $h$: 
$$ h: \R \ni  x\mapsto x+\alpha\sin(x) -p \in \R,\ -1<\alpha<1,$$
and we illustrate this case with an
\begin{enumerate}[\text{Case} (i)]
	\addtocounter{enumi}{4}
\item \label{item:seccaseex2}Ornstein Uhlenbeck process: $a_t=a>0$, $\beta_t=\beta>0$, $\gamma_t=0$, $\sigma_t=\sigma>0$, $X_0=x_0$ with $x_0 >|\alpha|+p$. We obtain
\begin{eqnarray*}
\d K_t =  e^{-at}\d \sup_{s \leq t} \l(F_s^{-1}(0)\r)^+,
\end{eqnarray*}
where for all $t$ in $[0,T]$,
\begin{eqnarray*}
F_t: \R \ni x &\mapsto& \Bigg\{e^{-at}\l(x_0-\beta\l(\frac{e^{at}-1}{a}\r) + x\r) + \alpha \exp\l(-e^{-at}\frac{\sigma^2}{a}\sinh(at)\r)\\
&& \times \sin\l(e^{-at}\l(x_0 - \beta\l(\frac{e^{at}-1}{a}\r) + x \r)\r)-p\Bigg\}.
\end{eqnarray*}
\end{enumerate}

\begin{remarque}
The reader may object that Case (\ref{item:firstcaseex3}) is out of the scope of our theoretical results, which is true. Nevertheless we let it in order to illustrate the robustness of the numerical method.
\end{remarque}

These examples have been chosen in such a way that we are able to give an analytic form of the reflecting process $K$. This enables us to compare numerically the ``true'' process $K$ and its empirical approximation $\hat{K}$. When an exact simulation of the underlying process is available, we compute the approximation rate of our algorithm. 

\subsection{Proofs of the numerical illustrations}
\label{sec:proofnumeric}
In order to have closed, or almost closed, expression for the compensator $K$ we introduce the process $Y$ solution to the non-reflected SDE
\begin{equation*}
Y_t = X_0-\int_0^t \left(\beta_s +  a_s Y_s\right)\d s + \int_0^t \left(\sigma_s + \gamma_s Y_s\right) \d B_s.
\end{equation*}
Letting $A_t = \int_0^t a_s \d s$ and applying Itô's formula on $e^{A_t}(X_t-Y_t)$, we get 
\begin{equation*}
X_t = Y_t + e^{-A_t}\int_{0}^t e^{A_s} \d K_s + e^{-A_t}\int_{0}^t e^{A_s}  \gamma_s (X_s-Y_s) \d B_s.
\end{equation*}
Hence, the constraint $\E[h(X_t)]\geq 0$ rewrites
\begin{equation}\label{eq:expressioncontrainte}
\E\left[h\left(Y_t + e^{-A_t}\int_{0}^t e^{A_s} \d K_s + e^{-A_t}\int_{0}^t e^{A_s} \gamma_s (X_s-Y_s)\d B_s\right)\right]\geq 0.
\end{equation}

\begin{proof}[Proof of assertions \eqref{item:firstcaseex1}, \eqref{item:firstcaseex2} and \eqref{item:firstcaseex4}] 
	The formula for $K$ comes from the expression of its density given in Corollary \ref{prop:densityK} and the fact that in all these cases
	\begin{equation*}
		\E\left[Y_t\right] - p = e^{-at} \left(x_0+\frac{\beta}{a}\right) - \left(p+\frac{\beta}{a}\right).
	\end{equation*}
\end{proof}

\begin{proof}[Proof of~\eqref{item:firstcaseex3}]
Recall that we supposed $h: \R \ni x \mapsto x-p \in \R$. In that case, since $\gamma\equiv 0$, the constraint \eqref{eq:expressioncontrainte} becomes
\begin{equation}\label{eq:contraintepos}
\E\left[e^{-A_t}\int_{0}^t e^{A_s} \d K_s\right] \geq   p -\E [Y_t],
\end{equation}
so that $K$ is nondecreasing with $K_0=0$ and, for all $t$ in $[0,T]$, 
\begin{equation}\label{eq:expressionkt}
\E\left[e^{-A_t}\int_{0}^t e^{A_s} \d K_s\right] \geq   p -\E [Y_t],\qquad \int_0^t \left(\E[X_s]-p\right) \, \d K_s = 0.
\end{equation}

Note first that since 
\begin{equation*}
\E \left[Y_s\right] = \E\left[e^{-A_s}\right]\, x_0 - \E\left[\int_0^s e^{-(A_s-A_r)} \beta \d r\right]  + \E \left[e^{-A_s} \int_0^s e^{A_{r}} \sigma \d B_r\right],
\end{equation*}
and using the integration by parts formula we have
\begin{equation*}
\E \left[ e^{-A_s} \int_0^s e^{A_{r}} \sigma \d B_r\right] = \sigma\,\E \left[\int_0^s D_r \left(e^{-A_s}\right)e^{A_{r}}  \d r\right] = -\sigma\, \E\left[ \int_0^s  \int_r^s (D_r a_u)\d u \, e^{-(A_s-A_r)} \d r\right].
\end{equation*}

Remember that, in this case, we supposed that $T\leq 1$, $\beta_t=\beta>0$ and $a_t=-\epsilon B_t$ for $\epsilon>0$ supposed to be small enough. We here illustrate the dependence of the processes w.r.t. the parameter $\epsilon$ by adding a superscript $\epsilon$ on $Y$ and $K$. Since $\int_s^t B_r\, dr$ is a centered gaussian random variable with variance $(t-s)^3/3$, we have
\begin{align*}
	\E\left[Y^\epsilon_t\right] & = x_0 \, e^{\epsilon^2 t^3/6} - \beta \int_0^t e^{\epsilon^2(t-s)^3/6}\, ds + \sigma\epsilon \, \int_0^t (t-s) e^{\epsilon^2(t-s)^3/6}\, ds, \\
	\E\left[X^\epsilon_t\right] & = \E\left[Y^\epsilon_t\right] + \int_0^t e^{\epsilon^2(t-s)^3/6}\, dK^\epsilon_s.
\end{align*}
Therefore, for all $s\leq t$,
\begin{equation*}
	\E\left[X^\epsilon_s\right] -p = x_0-p-\beta s + \epsilon\sigma\, \frac{s^2}{2} + o(\epsilon) + K^\epsilon_t \left(1+o(\epsilon)\right).
\end{equation*}
It follows that, up to $o(\epsilon)$,
\begin{equation*}
	K_t^\epsilon =  \sup_{s\leq t}\left( -(x_0-p) + \beta s - \sigma \epsilon  \frac{s^2}{2} \right)^+ .
\end{equation*}
Since $\epsilon\to 0^{+}$, we assume that $\beta^2 > 2\epsilon\sigma(x_0-p)$ and we obtain $K^\epsilon_t = 0$ if $t<t^\star$,
\begin{equation*}
	K^{\epsilon}_t = -(x_0-p) +\beta t - \sigma \epsilon  \frac{t^2}{2}, \text{ if } t^\star \leq t < \frac{\beta}{\epsilon\sigma}, \qquad K^{\epsilon}_t= -(x_0-p) + \frac{\beta^2}{2\epsilon\sigma}, \text{ for } t\geq \frac{\beta}{\epsilon\sigma},
\end{equation*}
where $t^\star = \left(\beta - \sqrt{\beta^2-2\epsilon\sigma(x_0-p)}\right)/(\epsilon\sigma)$.
\end{proof}

\begin{proof}[Proof of assertion \eqref{item:seccaseex2}]
In that case, we have
\begin{align*}
Y_t &= e^{-at}\l(x_0- \beta \l(\frac{e^{at}-1}{a}\r)\r) + \sigma e^{-at} \int_0^t e^{as}\d B_s := f_t +G_t,
\end{align*}
and 
\begin{eqnarray*}
X_t = Y_t + e^{-at}\bar{K}_t,\quad \bar{K}_t =\int_0^t e^{as}\d K_s.
\end{eqnarray*}
Hence 
\begin{eqnarray*}
h(X_t) &=&  Y_t +e^{-at} \bar{K}_t + \alpha \sin(Y_t+e^{-at}\bar{K}_t)-p\\
&=& Y_t +e^{-at} \bar{K}_t + \alpha \l(\sin(Y_t) \cos(e^{-at}\bar{K}_t)+\cos(Y_t) \sin(e^{-at}\bar{K}_t)\r)-p\\
&=& Y_t +e^{-at} \bar{K}_t + \alpha \Big[ \cos(e^{-at}\bar{K}_t)\l\{ \sin(f_t)\cos(G_t)+\sin(G_t)\cos(f_t) \r\}\\
&& +\sin(e^{-at}\bar{K}_t)\l\{ \cos(f_t)\cos(G_t)-\sin(f_t)\sin(G_t)\r\}\Big]-p.
\end{eqnarray*}
Since $G_t$ is a centered gaussian random variable with variance $\sigma^2\dfrac{1-e^{-2at}}{2a}=\sigma^2e^{-at}\dfrac{\sinh(at)}{a}$,
\begin{equation*}
\E[\sin(G_t)] = 0,\quad \text{   and    }\quad  \E[\cos(G_t)] = \exp\l(-e^{-at}\frac{\sigma^2}{a}\sinh(at)\r)=:g(t),
\end{equation*}
we obtain that 
\begin{align*}
\E \left[h(X_t)\right] &= f(t) + e^{-at}\bar{K}_t + \alpha g(t)\sin\l(f_t + e^{-at}\bar{K}_t\r)-p, \\
& :=F_t(\bar{K}_t).
\end{align*}
Therefore,
\begin{equation*}
\bar{K}_t = \sup_{s \leq t}\l(F_s^{-1}(0)\r)^+, \quad\text{ and}\quad \d K_t = e^{-at} \d \sup_{s \leq t} \l(F_s^{-1}(0)\r)^+.
\end{equation*}
\end{proof}

\subsection{Illustrations}

This computation works as follows. Let $0=T_0 < T_1 <\cdots<T_{n}=T$ be a subdivision of $[0,T]$ of step size $1/n$, $n$ being a positive integer, let $X$ be the unique solution of the MRSDE \eqref{eq:mainNumIll} and let, for a given $i$, $\big(\tX_{T_k}^{i}\big)_{0\leq k \leq n}$ be its numerical approximation given by Algortihm \ref{algo:KF}. For a given integer $L$, we draw $(\bX^l)_{1\leq l \leq L}$ and $\left(\tX^{i,l}\right)_{1\leq l \leq L}$, $L$ independent copies of $X$ and $\tX^{i}$. We then approximate the $\mathbb{L}^2$-error of Proposition \ref{thm:NS} by:
\begin{equation}
\hat{E} = \left(\frac{1}{L}\sum_{l=1}^L \max_{0\leq k \leq n} \l|X_{T_k}^l - \tX_{T_k}^{i,l}\r|^2\right)^{1/2}.
\end{equation}

\begin{figure}[h!]
\includegraphics[scale=0.5]{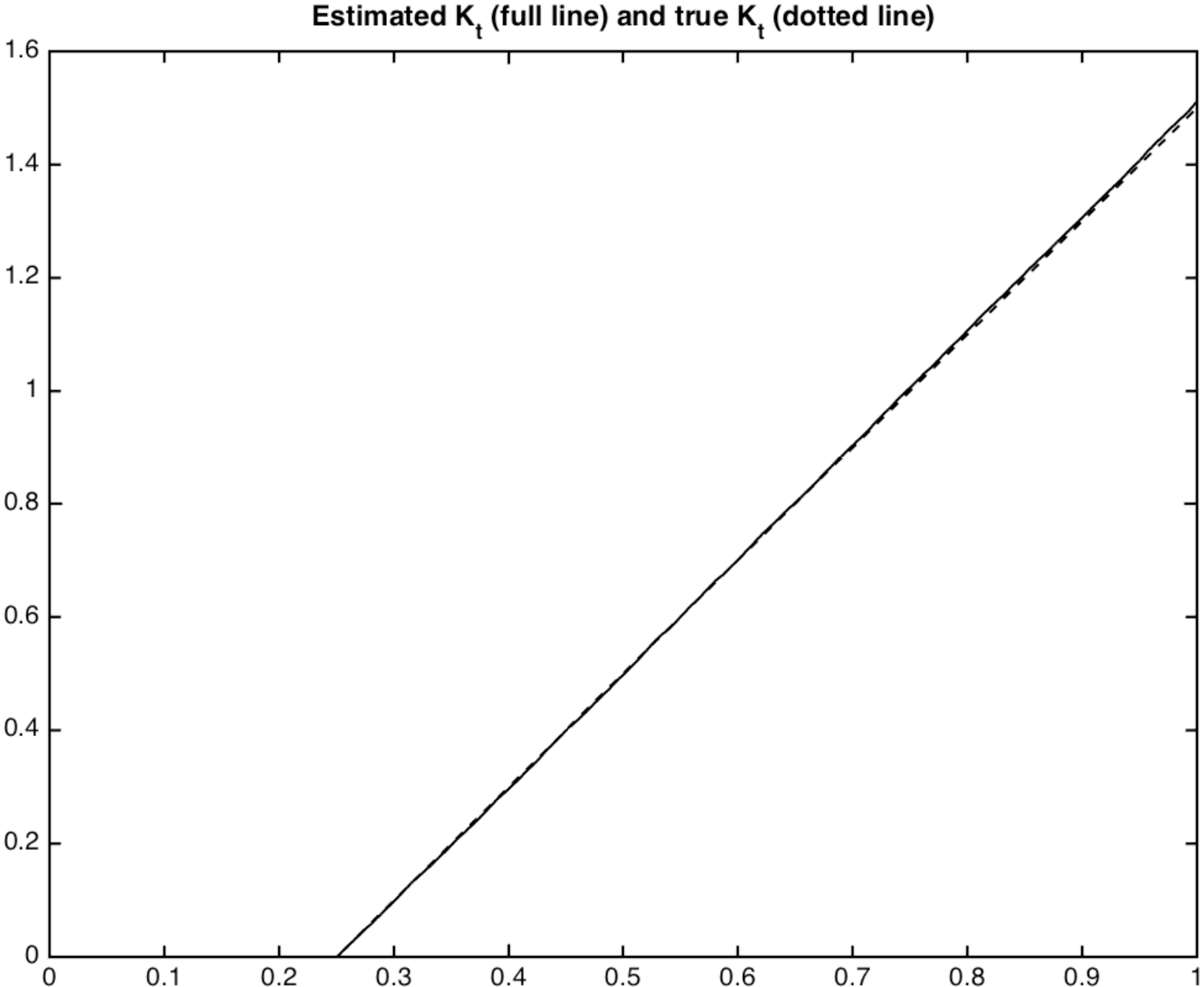}
\caption{Case \eqref{item:firstcaseex1}. $n=500,\ N=10000,\ T=1,\ \beta=2,\ \sigma=1,\ x_0=1,\ p=1/2$ }
\label{fig1}
\end{figure}
\begin{figure}[h!]
\includegraphics[scale=0.5]{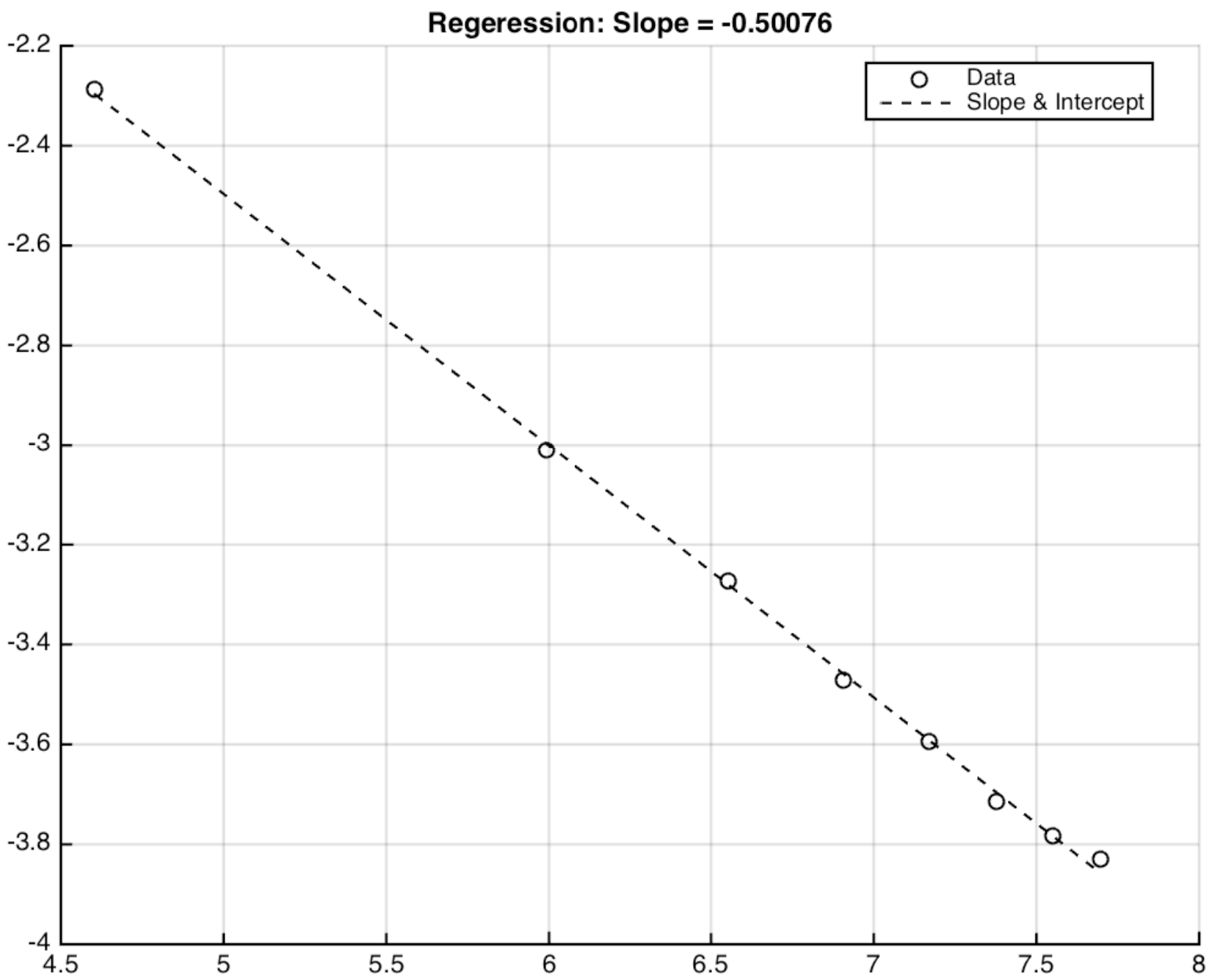}
\caption{Case \eqref{item:firstcaseex1}. Regression of $(1/2)\log(\hat{E})$ wr.r.t. $\log(n)$. Data: $\hat{E}$ when $n$ varies from $100$ to $2200$ with step size $300$. Parameters: $\ N=1000,\ T=1,\ \beta=2,\ \sigma=1,\ x_0=1,\ p=1/2$, $L =1000$. }
\label{fig2}
\end{figure}
\begin{figure}[h!]
\includegraphics[scale=0.5]{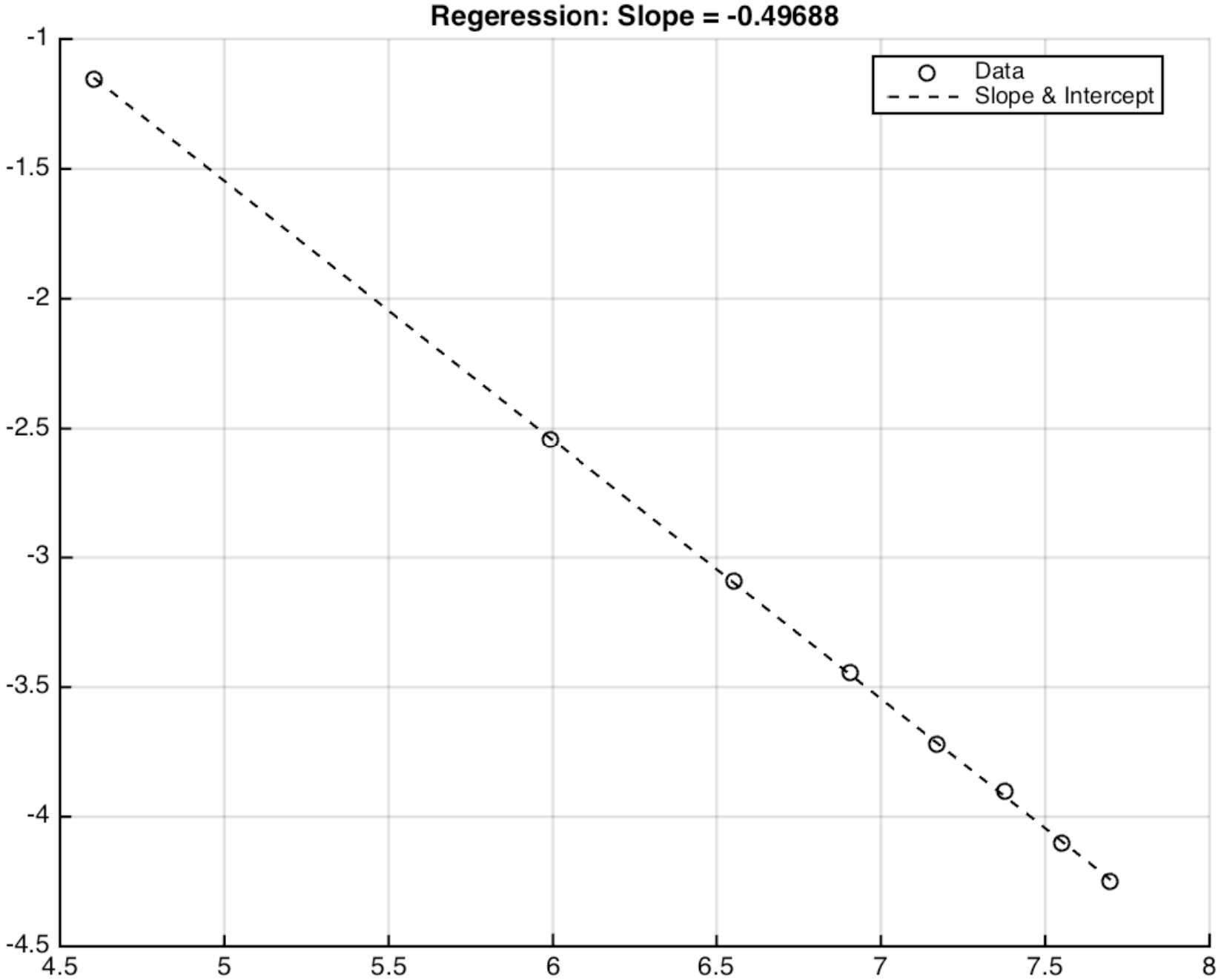}
\caption{Case \eqref{item:firstcaseex1}. Regression of $(1/2)\log(\hat{E})$ wr.r.t. $\log(N)$. Data: $\hat{E}$ when $N$ varies from $100$ to $2200$ with step size $300$. Parameters: $\ n=100,\ T=1,\ \beta=2,\ \sigma=1,\ x_0=1,\ p=1/2$, $L =1000$. }
\label{fig3}
\end{figure}

\begin{figure}[h!]
\includegraphics[scale=0.5]{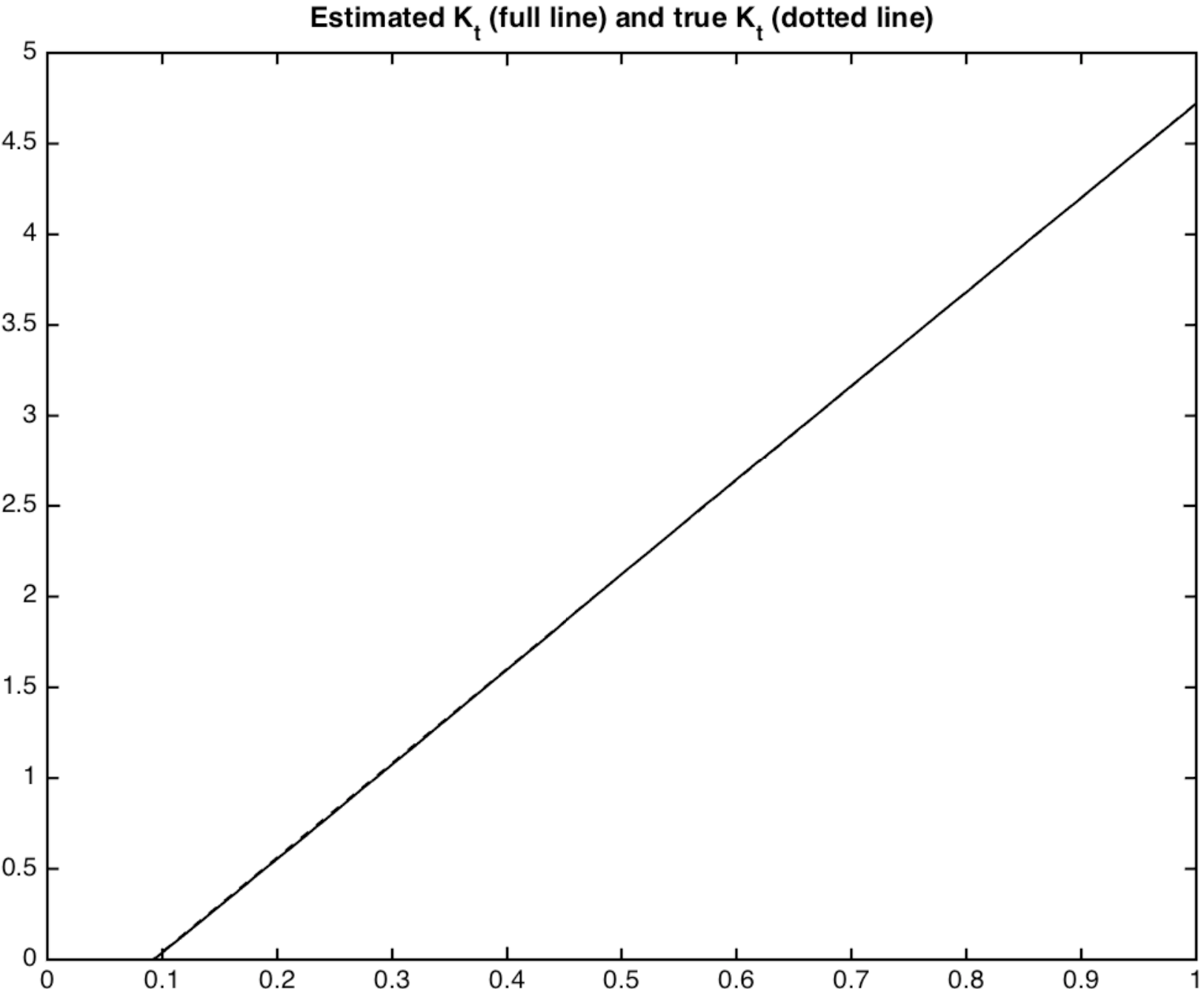}
\caption{Case \eqref{item:firstcaseex2}. $n=500,\ N=10000,\ T=1,\ \beta=2.1,\ a=1,\ \sigma=1,\ x_0=1,\ p=3.6$ }
\label{fig4}
\end{figure}
\begin{figure}[h!]
\includegraphics[scale=0.5]{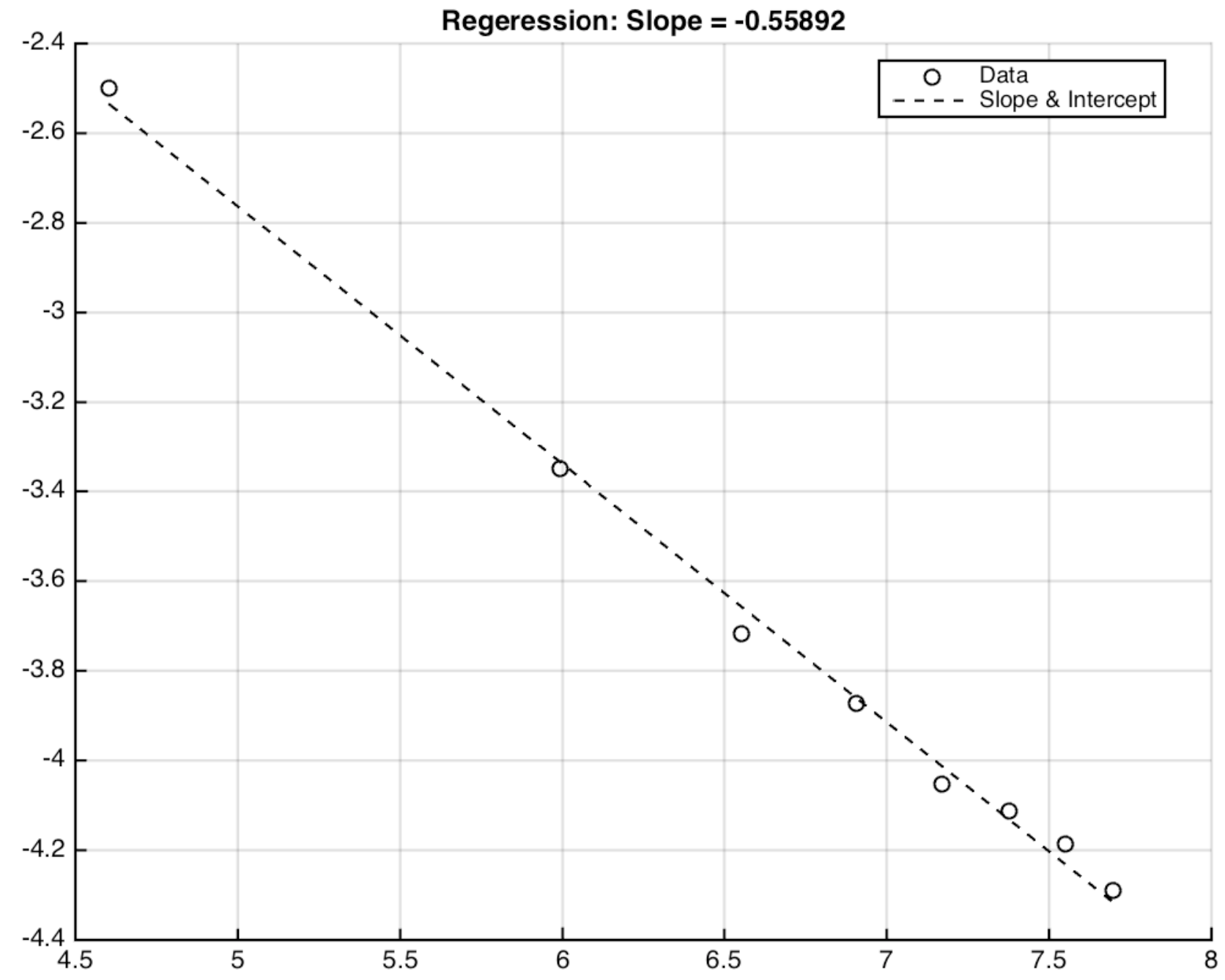}
\caption{Case \eqref{item:firstcaseex2}. Regression of $(1/2)\log(\hat{E})$ wr.r.t. $\log(n)$. Data: $\hat{E}$ when $n$ varies from $100$ to $2200$ with step size $300$. Parameters: $\ N=1000,\ T=1,\ \beta=2,\ \sigma=1,\ x_0=1,\ p=1/2$, $L =1000$. }
\label{fig5}
\end{figure}
\begin{figure}[h!]
\includegraphics[scale=0.5]{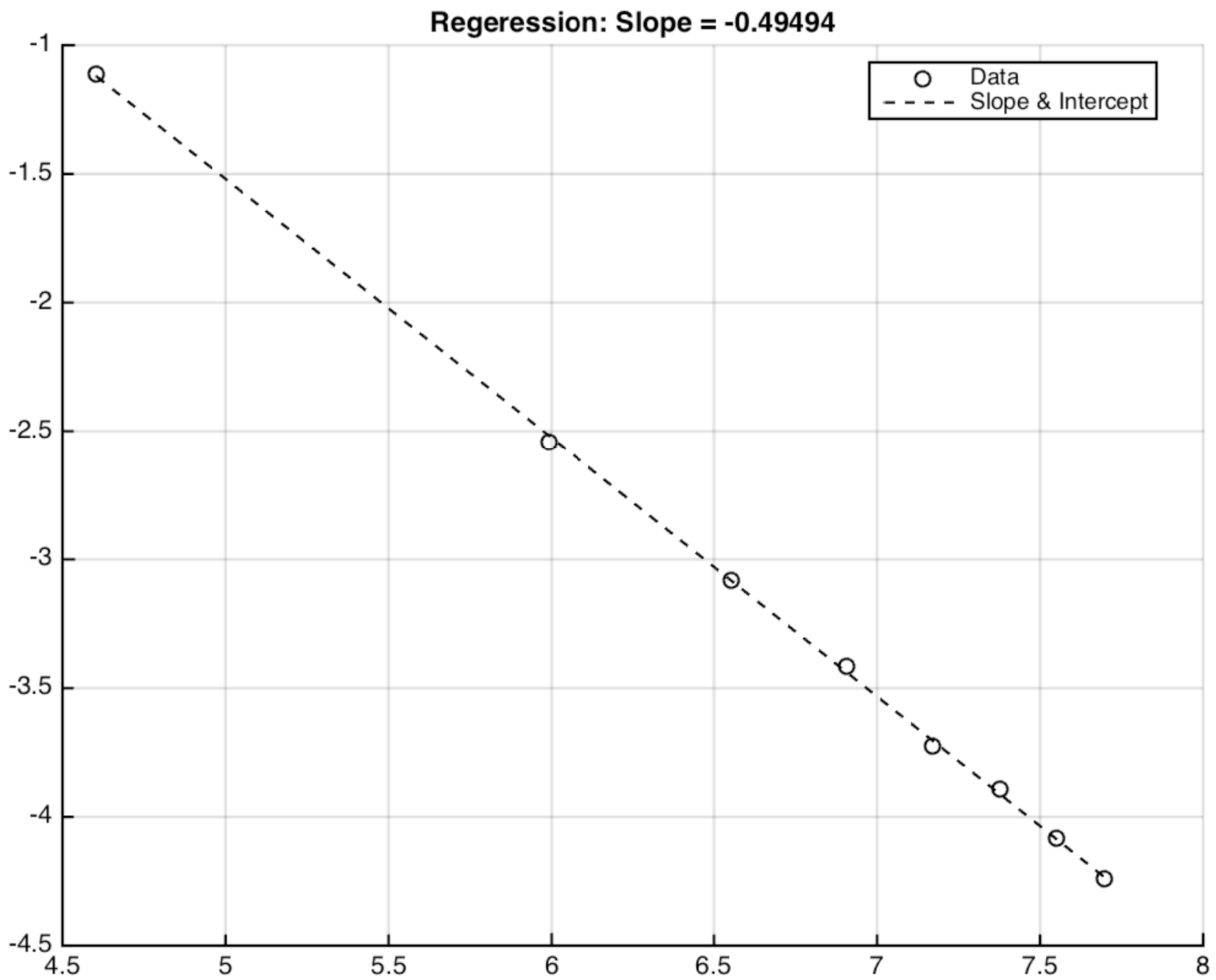}
\caption{Case \eqref{item:firstcaseex2}. Regression of $(1/2)\log(\hat{E})$ wr.r.t. $\log(N)$. Data: $\hat{E}$ when $N$ varies from $100$ to $2200$ with step size $300$. Parameters: $\ n=100,\ T=1,\ \beta=2,\ \sigma=1,\ x_0=1,\ p=1/2$, $L =1000$. }
\label{fig6}
\end{figure}

\begin{figure}[h!]
\includegraphics[scale=0.5]{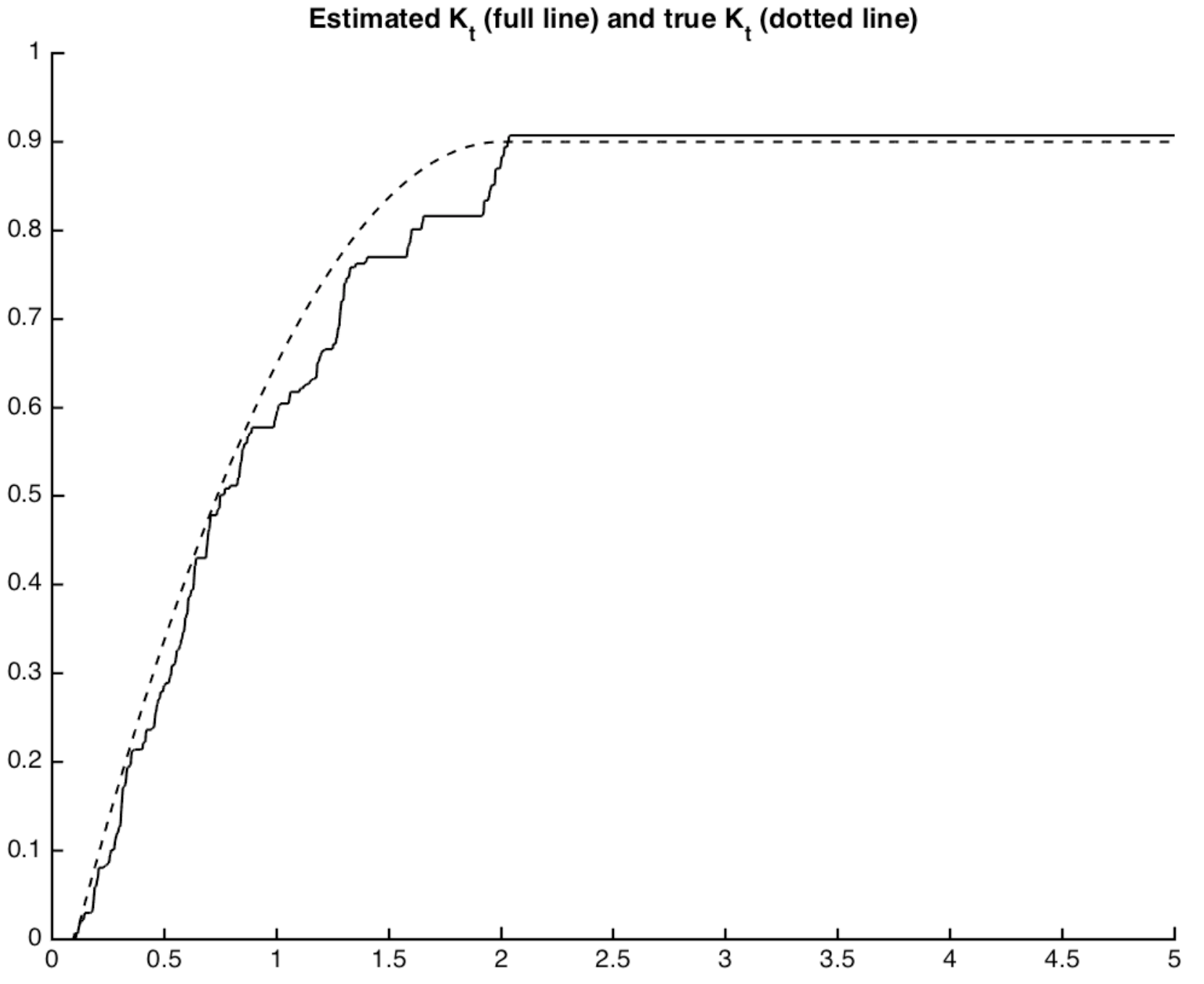}
\caption{Case \eqref{item:firstcaseex3}. Parameters: $n=2000,\ N=10000,\ T=5,\ \beta=1,\ \epsilon=5/100,\ \sigma=1/(2\epsilon),\ x_0=1,\ p=0.9$ }
\label{fig7}
\end{figure}

\begin{figure}[h!]
\includegraphics[scale=0.5]{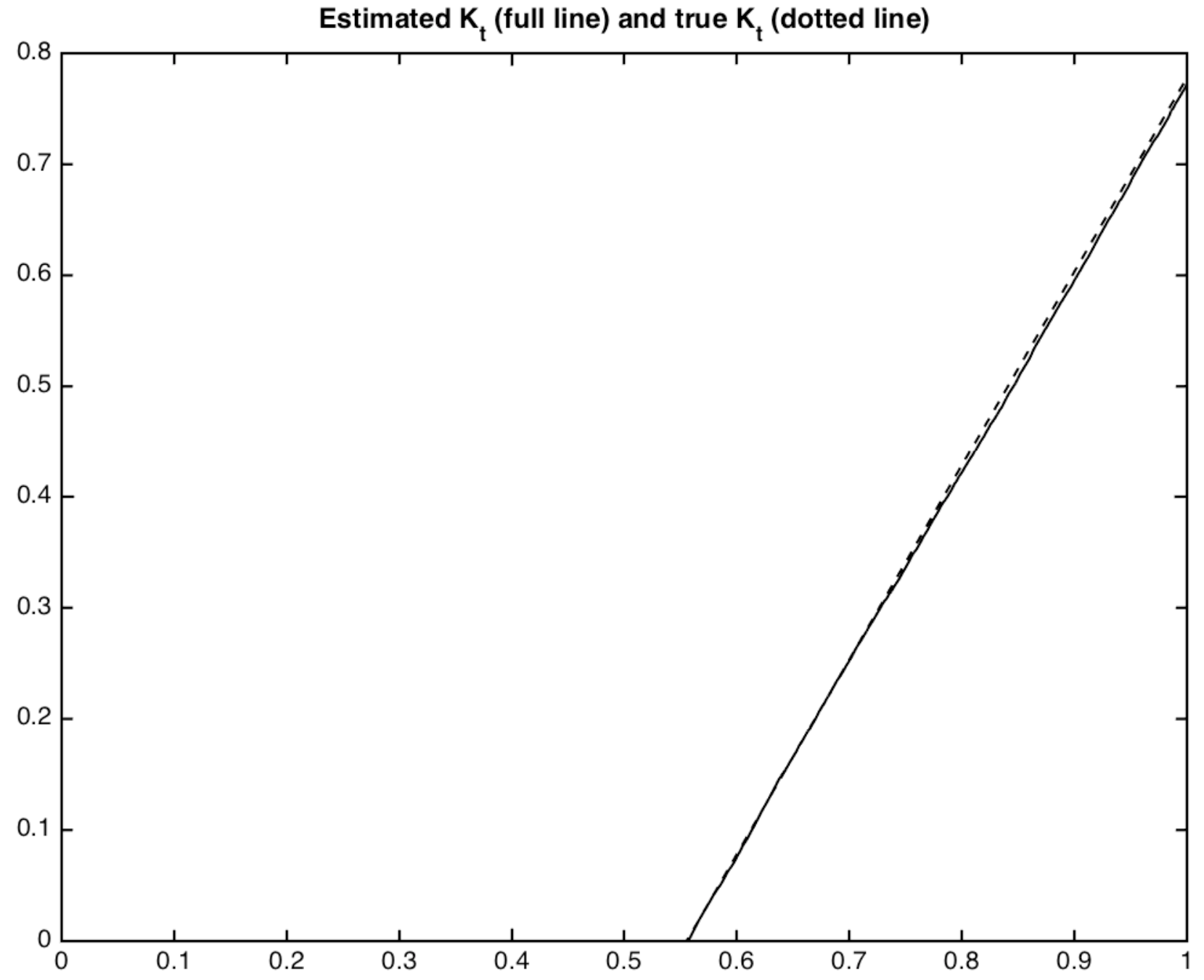}
\caption{Case \eqref{item:firstcaseex4}. Parameters: $n=500,\ N=10000,\ T=1,\ a=1,\ \gamma=1,\ x_0=4,\ p=1$ }
\label{fig8}
\end{figure}

\begin{figure}[h!]
\includegraphics[scale=0.5]{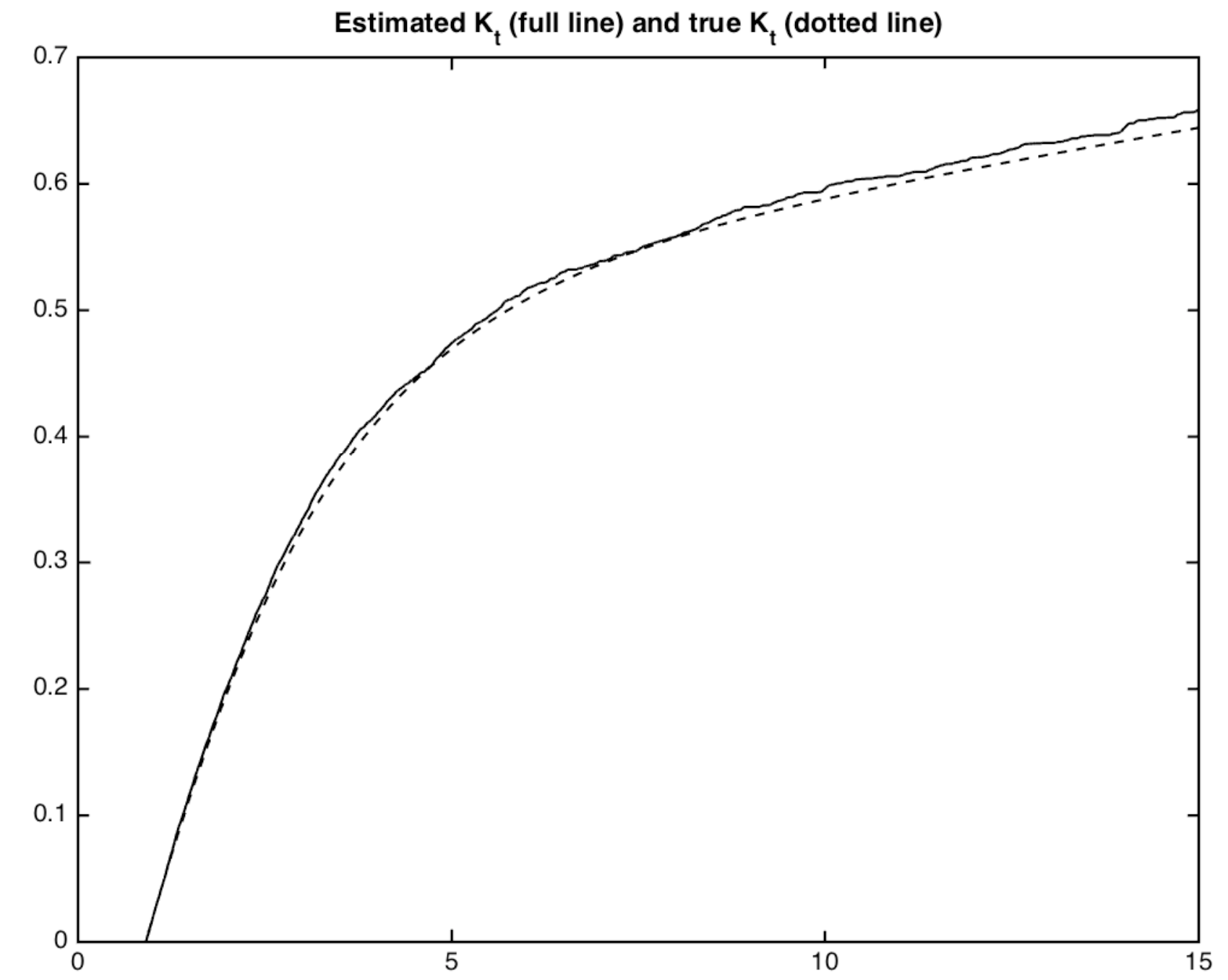}
\caption{Case \eqref{item:seccaseex2}. Parameters: $n=1000,\ N=100000,\ T=15,\ \beta=10^{-2},\ \sigma=1,\ p=\pi/2,\ \alpha=0.9,\ x_0$ is the unique solution of $x+\alpha \sin(x)-p=0$ plus $10^{-1}$. }
\label{fig9}
\end{figure}

\clearpage

\bibliographystyle{alpha}
\bibliography{BibSDEMR}

\end{document}